\pdfoutput=1
\RequirePackage{ifpdf}
\ifpdf 
\documentclass[pdftex]{sigma}
\else
\documentclass{sigma}
\fi

\numberwithin{equation}{section}

\newtheorem{Theorem}{Theorem}[section]
\newtheorem*{Theorem*}{Theorem}
\newtheorem{Corollary}[Theorem]{Corollary}
\newtheorem{Lemma}[Theorem]{Lemma}
\newtheorem{Claim}[Theorem]{Claim}

\theoremstyle{definition}
\newtheorem{Definition}[Theorem]{Definition}

\newtheorem{Remark}[Theorem]{Remark}

\DeclareMathOperator{\ev}{ev}
\DeclareMathOperator{\id}{id}

\DeclareMathOperator{\ad}{ad}

\DeclareMathOperator{\row}{row}
\DeclareMathOperator{\col}{col}

\newcommand{\ve}{\varepsilon}

\begin{document}

\allowdisplaybreaks

\newcommand{\arXivNumber}{2407.15367}

\renewcommand{\PaperNumber}{007}

\FirstPageHeading

\ShortArticleName{Commuting Subalgebras of Affine Super Yangians Arising from Edge Contractions}

\ArticleName{Commuting Subalgebras of Affine Super Yangians\\ Arising from Edge Contractions}

\Author{Mamoru UEDA}

\AuthorNameForHeading{M.~Ueda}

\Address{Graduate School of Mathematical Sciences, The University of Tokyo, \\
3-8-1 Komaba Meguro-ku Tokyo 153-8914, Japan}
\Email{\mail{mueda@ms.u-tokyo.ac.jp}}

\ArticleDates{Received August 08, 2025, in final form January 08, 2026; Published online January 31, 2026}

\Abstract{In the previous paper, we constructed two kinds of edge contractions for the affine super Yangian and a homomorphism from the affine super Yangian to the universal enveloping algebra of a $W$-superalgebra of type $A$. In this article, we show that these two edge contractions commute with each other. As an application, we give a homomorphism from the affine super Yangian to some centralizer algebras of the universal enveloping algebra of $W$-superalgebras of type $A$. Using the edge contraction, we also show the compatibility of the coproduct for the affine super Yangian with the parabolic induction for a $W$-superalgebra of type $A$ in some special cases.}


\Classification{17B37; 17B69}

\section{Introduction}

The Yangian $Y_\hbar(\mathfrak{g})$ associated with a finite dimensional simple Lie algebra $\mathfrak{g}$ was introduced by Drinfeld \cite{D1, D2}. The Yangian $Y_\hbar(\mathfrak{g})$ is a quantum group which is a deformation of the current algebra $\mathfrak{g}\otimes\mathbb{C}[z]$. The Yangian of type $A$ has several presentations: the RTT presentation, the current presentation, the parabolic presentation and so on.
By using the current presentation, we can extend the definition of the Yangian $Y_\hbar(\mathfrak{g})$ to a symmetrizable Kac--Moody Lie algebra~$\mathfrak{g}$. Especially, in the case that $\mathfrak{g}$ is of affine type, Guay--Nakajima--Wendlandt \cite{GNW} defined the coproduct for the affine Yangian.

One of the difference between finite Yangians of type $A$ and affine Yangians of type $A$ is the existence of the RTT presentation and the parabolic presentation (see \cite{BK2}). By using these presentations, two embeddings were constructed for the finite Yangian:
\begin{align*}
\Psi_1^f\colon\ Y(\mathfrak{gl}(n))\to Y(\mathfrak{gl}(m+n)),\qquad
\Psi_2^f\colon\ Y(\mathfrak{gl}(m))\to Y(\mathfrak{gl}(m+n)),
\end{align*}
where $Y(\mathfrak{gl}(n))$ is the Yangian associated with $\mathfrak{gl}(n)$.
By using \smash{$\Psi_1^f$} and \smash{$\Psi_2^f$}, Olshanskii \cite{Ol} gave a homomorphism from the finite Yangian \smash{$Y(\mathfrak{gl}(m))$} to the centralizer algebra of $U(\mathfrak{gl}(m+n))$ and $U(\mathfrak{gl}(n))$. Moreover, $Y(\mathfrak{gl}(m))$ can be embedded into the projective limit of this centralizer algebra.
In \cite{U10}, we gave the affine version of \smash{$\Psi_1^f$} and \smash{$\Psi_2^f$} and constructed a homomorphism from the affine Yangian associated with \smash{$\widehat{\mathfrak{sl}}(m)$} to the centralizer algebra of \smash{$U\big(\widehat{\mathfrak{gl}}(m+n)\big)$} and~\smash{$U\big(\widehat{\mathfrak{gl}}(n)\big)$}.

In super setting, Nazarov \cite{Na} introduced the Yangian associated with $\mathfrak{gl}(m|n)$ by using the RTT presentation and Stukopin \cite{S} defined the Yangian of $\mathfrak{sl}(m|n)$ by using the current presentation. Peng \cite{Pe} gave a parabolic presentation of the super Yangian associated with~$\mathfrak{gl}(m|n)$. The author \cite{U2} defined the affine super Yangian associated with \smash{$\widehat{\mathfrak{sl}}(m|n)$} as a~quantum group.\looseness=1

In \cite{U13}, we gave two homomorphisms called the edge contractions for the affine super Yangian:
\begin{align*}
\begin{split}
&\Psi_1\colon\ Y_{\hbar,\ve}\big(\widehat{\mathfrak{sl}}(m_1|n_1)\big)\to \widetilde{Y}_{\hbar,\ve}\big(\widehat{\mathfrak{sl}}(m_1+m_2|n_1+n_2)\big),\\
&\Psi_2\colon\ Y_{\hbar,\ve+(m_1-n_1)\hbar}\big(\widehat{\mathfrak{sl}}(m_2|n_2)\big)\to \widetilde{Y}_{\hbar,\ve}\big(\widehat{\mathfrak{sl}}(m_1+m_2|n_1+n_2)\big),
\end{split}
\end{align*}
where \smash{$\widetilde{Y}_{\hbar,\ve}\big(\widehat{\mathfrak{sl}}(m_1+m_2|n_1+n_2)\big)$} is the standard degreewise completion of $Y_{\hbar,\ve}\smash{\big(\widehat{\mathfrak{sl}}(m_1+m_2}|n_1\allowbreak +n_2)\big)$.

The main theorem of this article is the following.

\begin{Theorem}\label{A}
The images of $\Psi_1$ and $\Psi_2$ commute with each other.
\end{Theorem}
By Theorem~\ref{A}, we obtain a homomorphism
\begin{equation*}
\Psi_1\otimes\Psi_2\colon\ Y_{\hbar,\ve}\big(\widehat{\mathfrak{sl}}(m_1|n_1)\big)\otimes Y_{\hbar,\ve+(m_1-n_1)\hbar}\big(\widehat{\mathfrak{sl}}(m_2|n_2)\big)\to \widetilde{Y}_{\hbar,\ve}\big(\widehat{\mathfrak{sl}}(m_1+m_2|n_1+n_2)\big).
\end{equation*}
The quantum toroidal algebra often has the same result as the affine Yangian. For example, the evaluation map for the quantum toroidal algebra was given by Miki \cite{M1} in the non-super setting and by Bezzera--Muhkin \cite{BM} in the super setting. The non-super version of Theorem~\ref{A} was given for the quantum toroidal algebra by Feigin--Jimbo--Miwa--Muhkin \cite{FJMM}, which corresponds to the author's work \cite{U10}. In the quantum toroidal setting, the proof was given by using the relations of the current presentation. Unfortunately, the corresponding relations are not given in the affine Yangian setting. Thus, we prove Theorem~\ref{A} by using the finite presentation.

As an application of Theorem~\ref{A}, we can give a relationship between the affine super Yangian and a centralizer algebra of $U\big(\widehat{\mathfrak{gl}}(m|n)\big)$.
For an associative superalgebra $A$ and its subalgebra~$B$, we set
\begin{equation*}
C(A,B)=\{x\in A\mid [x,y]=0\text{for }y\in B\}.
\end{equation*}
The affine super Yangian has a surjective homomorphism called the evaluation map \cite{U3, U2}:
\begin{equation*}
\ev^{m|n}_{\hbar,\ve}\colon\ Y_{\hbar,\ve}\big(\widehat{\mathfrak{sl}}(m|n)\big)\to U\big(\widehat{\mathfrak{gl}}(m|n)\big).
\end{equation*}
By combining \smash{$\ev^{m_1+m_2|n_1+n_2}_{\hbar,\ve}$} and $\Psi_2$, we obtain a homomorphism
\begin{gather*}
\ev^{m_1+m_2|n_1+n_2}_{\hbar,\ve}\circ\Psi_2\colon\\
\qquad Y_{\hbar,\ve+(m_1-n_1)\hbar}\big(\widehat{\mathfrak{sl}}(m_2|n_2)\big)\to C\big(U\big(\widehat{\mathfrak{gl}}(m_1+m_2|n_1+n_2)\big),U\big(\widehat{\mathfrak{gl}}(m_1|n_1)\big)\big).
\end{gather*}
Similarly to finite setting, we expect that the affine super Yangian can be embedded into the projective limit of the centralizer algebra \smash{$C\big(U\big(\widehat{\mathfrak{gl}}(m_1+m_2|n_1+n_2)\big),U\big(\widehat{\mathfrak{gl}}(m_1|n_1)\big)\big)$} through this~homomorphism. We also conjecture that \smash{$C\big(U\big(\widehat{\mathfrak{gl}}(m_1+m_2|n_1+n_2)\big),U\big(\widehat{\mathfrak{gl}}(m_1|n_1)\big)\big)$} is isomorphic to the tensor product of the center of \smash{$U\big(\widehat{\mathfrak{gl}}(m_1|n_1)\big)$} and the image of \smash{$\ev^{m_1+m_2|n_1+n_2}_{\hbar,\ve}\circ\Psi_2$}.

The similar result holds for $W$-superalgebras of type $A$. A $W$-superalgebra $\mathcal{W}^k(\mathfrak{g},f)$ is a~vertex superalgebra associated with a finite dimensional reductive Lie superalgebra $\mathfrak{g}$, an even nilpotent element $f$ and a complex number $k$. Let us set
\begin{align*}
&M=\sum\limits_{i=1}^lu_i,\qquad u_1\geq u_{2}\geq\dots\geq u_l\geq u_{l+1}=0,\nonumber\\
&N=\sum\limits_{i=1}^lq_i,\qquad q_1\geq q_{2}\geq\dots\geq q_l\geq q_{l+1}=0
\end{align*}
and assume that $u_l+q_l\neq0$ and $M\neq N$. Let us take \smash{$f\in\mathfrak{gl}(M|N)=\bigoplus_{i,j\in I_{M|N}}\mathbb{C}E_{i,j}$} as a~nilpotent element of type \smash{$\big(1^{u_1-u_2|q_1-q_2},2^{u_2-u_3|q_2-q_3},\dots,l^{u_l-u_{l+1}|q_l-q_{l+1}}\big)$}. In \cite{U13}, the author has given a homomorphism
\begin{equation*}
\Phi_s\colon\ Y_{\hbar,\ve}\big(\widehat{\mathfrak{sl}}(u_s-u_{s+1}|q_s-q_{s+1})\big)\to\mathcal{U}\big(\mathcal{W}^k(\mathfrak{gl}(M|N),f)\big),
\end{equation*}
where $\mathcal{U}\big(\mathcal{W}^k(\mathfrak{gl}(M|N),f)\big)$ is the universal enveloping algebra of $\mathcal{W}^k(\mathfrak{gl}(M|N),f)$. By Theorem~\ref{A}, we find that $\{\Psi_s\}$ commute with each other. In the case that $u_1=u_2=\dots=u_l=m$ and $q_1=q_2=\dots=q_l=n$, we call $\mathcal{W}^k(\mathfrak{gl}(M|N),f)$ the rectangular $W$-superalgebra of type~$A$ and denote it by \smash{$\mathcal{W}^k\big(\mathfrak{gl}(ml|nl),\big(l^{m|n}\big)\big)$}. In the rectangular case, we \cite{U4} showed that $\Phi_1$ is surjective.
In rectangular setting, there exists a natural embedding from $\mathcal{W}^{k+m_2-n_2}\big(\mathfrak{gl}(2m_1|2n_1),\allowbreak \big(2^{m_1|n_1}\big)\big)$ to \smash{$\mathcal{W}^k\big(\mathfrak{gl}(2m_1+2m_2|2n_1+2n_2),\big(2^{m_1+m_2|n_1+n_2}\big)\big)$}. By Theorem~\ref{A}, we obtain a homomorphism
\begin{align*}
\Phi_1\circ\Psi_2\colon\ Y_{\hbar,\ve+(m_1-n_1)\hbar}\big(\widehat{\mathfrak{sl}}(m_2|n_2)\big)\to C(\mathcal{U}(W_1),\mathcal{U}(W_2)).
\end{align*}
where
\begin{align*}
&W_1=\mathcal{W}^k\big(\mathfrak{gl}(2m_1+2m_2|2n_1+2n_2),\big(2^{m_1+m_2|n_1+n_2}\big)\big),\\
&W_2=\mathcal{W}^{k+m_2-n_2}\big(\mathfrak{gl}(2m_1|2n_1),\big(2^{m_1|n_1}\big)\big).
\end{align*}

As for non-rectangular cases, if $u_1>u_2$, $q_1>q_2$ (resp.\ $u_1=u_2>u_3$, $q_1=q_2>q_3$), the image of $\Phi_1\circ\Psi_1$ (resp.\ $\Phi_1\circ\Psi_2$) coincides with \smash{$U\big(\widehat{\mathfrak{gl}}(u_1-u_2|q_1-q_2)\big)$} \big(resp.\ the rectangular $W$-algebra associated with $\mathfrak{gl}(2u_1-2u_3|2q_1-2q_3)$ and a nilpotent element of type $\big(2^{u_1-u_3|q_1-q_3}\big)$\big). Then, $\Psi_s$ induces a homomorphism from the affine super Yangian to the centralizer algebra of $\mathcal{U}\big(\mathcal{W}^k(\mathfrak{gl}(M|N)),f\big)$
and \smash{$U\big(\widehat{\mathfrak{gl}}(u_1-u_2|q_1-q_2)\big)$} \big(resp.\ $\mathcal{U}\big(\mathcal{W}^k(\mathfrak{gl}(2(u_1-u_3)|2(q_1-q_3)),\allowbreak (2^{u_1-u_3|q_1-q_3}))\big)$\big).

We expect that this result can be applicable to the generalization of the Gaiotto--Rapcak's triality.
Gaiotto and Rapcak \cite{GR} introduced a kind of vertex algebras called $Y$-algebras and conjectured a triality of the isomorphism of $Y$-algebras. Let $f_{n,m}\in\mathfrak{sl}(m+n)$ be a nilpotent element of type $\bigl(n^1,1^m\bigr)$. It is known that some kinds of $Y$-algebras can be realized as a coset of the pair of $\mathcal{W}^k(\mathfrak{sl}(m+n),f_{n,m})$ and $V^{k-m-1}(\mathfrak{gl}(m))$ up to Heisenberg algebras. In this case, Creutzig--Linshaw \cite{CR} have proved the triality conjecture. This result is the generalization of the Feigin--Frenkel duality \cite{FF} and the coset realization of principal $W$-algebra.
The $Y$-algebras can be interpreted as a truncation of $\mathcal{W}_{1+\infty}$-algebra \cite{GG}, whose universal enveloping algebra is isomorphic to the affine Yangian of $\widehat{\mathfrak{gl}}(1)$ up to suitable completions (see \cite{AS,MO, T}).

For a vertex algebra $A$ and its vertex subalgebra $B$, let us set the coset vertex algebra of the~pair $A$ and $B$
${\rm Com}(A,B)=\{a\in A\mid b_{(r)}a=0\text{ for }r\geq0,\, b\in B\}$.
The homomorphism $\Phi_1\circ\Psi_2$ induces the one from the affine super Yangian \smash{$Y_{\hbar,\ve+(m_1-n_1)\hbar}\big(\widehat{\mathfrak{sl}}(m_2|n_2)\big)$} to the universal enveloping algebra of \smash{${\rm Com}\big(W_1,\mathcal{W}^k\big(\mathfrak{sl}(2m_1|2n_1),\big(2^{m_1|n_1}\big)\big)\big)$}. We expect that this homomorphism becomes surjective and induces the isomorphism
\begin{align*}
{\rm Com}(W_3,W_4)&\simeq {\rm Com}(W_5,W_6),
\end{align*}
where{\samepage
\begin{align*}
\begin{split}
&W_3=\mathcal{W}^k\big(\mathfrak{gl}(2m_1+2m_3|2n_1+2n_3),\big(2^{m_1+m_3|n_1+n_3}\big)\big),\\
&W_4=\mathcal{W}^{k+m_3-n_3}\big(\mathfrak{sl}(2m_1|2n_1),\big(2^{m_1|n_1}\big)\big),\\
&W_5=\mathcal{W}^k\big(\mathfrak{gl}(2m_2+2m_3|2n_2+2n_3),\big(2^{m_2+m_3|n_2+n_3}\big)\big),\\
&W_6=\mathcal{W}^{k+m_3-n_3}\big(\mathfrak{sl}(2m_2|2n_2),\big(2^{m_2|n_2}\big)\big).
\end{split}
\end{align*}
These are the generalizations of the Gaiotto--Rapcak's triality.}

For non-rectangular cases, we also expect that similar isomorphisms will hold. In order to consider the non-rectangular setting, we need to construct a relationship between the shifted affine super Yangian and a $W$-superalgebra of type $A$.
In the finite setting, Peng \cite{Pe} wrote down a finite $W$-superalgebra of type $A$ as a quotient algebra of the shifted super Yangian by using the parabolic presentation. Similarly to \cite{Pe}, it is conjectured that there exists a surjective homomorphism from the shifted affine super Yangian to the universal enveloping algebra of \mbox{a~$W$-superalgebra} of type $A$ if we change the definition of the shifted affine super Yangian properly. The image of $\Psi_1\otimes\Psi_2$ corresponds to the Levi subalgebra of the finite super Yangian of type~$A$, which is defined by the parabolic presentation. We expect that $\Psi_1\otimes\Psi_2$ will lead to a new definition of the shifted affine super Yangian.

In Sections~\ref{section7} and~\ref{section8},
we construct the parabolic induction for a $W$-superalgebra in the case that $u_1>u_2>\dots>u_l$, $q_1>q_2>\dots>q_l$ and show that the coproduct for the affine super Yangian is compatible with the parabolic induction via $\Phi_l$ in this case. In order to show the compatibility, by using the edge contraction $\Phi_2$, we need to extend the affine super Yangian. We expect that this extended affine super Yangian will be connected the new definition of the shifted affine super Yangian.

\section{Affine super Yangian}\label{section2}
Let us take integers $m,n\geq 2$ and $m+n\geq 5$. We set
\begin{align*}
I_{m|n}&=\{1,2,\dots,m,-1,-2,\dots,-n\}
\end{align*}
and define the parity on $I_{m|n}$ by
\[p(i)=\begin{cases}
0&\text{if }i>0,\\
1&\text{if }i<0.
\end{cases}\]
Sometimes, we identify $I_{m|n}$ with $\mathbb{Z}/(m+n)\mathbb{Z}$ by corresponding \smash{$-i\in I_{m|n}$} to $m+i\in\mathbb{Z}/(m+n)\mathbb{Z}$ for $1\leq i\leq n$. We set two matrices \smash{$(a_{i,j})_{i,j\in\mathbb{Z}/(m+n)\mathbb{Z}}$} and \smash{$(b_{i,j})_{i,j\in\mathbb{Z}/(m+n)\mathbb{Z}}$} as
\begin{equation*}
a_{i,j}=\begin{cases}
(-1)^{p(i)}+(-1)^{p(i+1)}&\text{if } i=j,\\
-(-1)^{p(i+1)}&\text{if }j=i+1,\\
-(-1)^{p(i)}&\text{if }j=i-1,\\
0&\text{otherwise},
\end{cases}\qquad
b_{i,j}=\begin{cases}
a_{i,j}&\text{if }j=i-1,\\
-a_{i,j}&\text{if }j=i+1,\\
0&\text{ otherwise}.
\end{cases}
\end{equation*}
\begin{Definition}[{\cite[Definition 3.1]{U2}}]
Let $\ve_1,\ve_2\in\mathbb{C}$. The affine super Yangian \smash{$Y_{\ve_1,\ve_2}\big(\widehat{\mathfrak{sl}}(m|n)\big)$} is the associative superalgebra over $\mathbb{C}$ generated by
\begin{equation*}
\bigl\{X_{i,r}^\pm, H_{i,r}\mid i\in I_{m|n}=\mathbb{Z}/(m+n)\mathbb{Z},\,r=0,1\bigr\}
\end{equation*}
subject to the following defining relations:
\begin{alignat*}{3}
&\makebox[0pt][l]{$[H_{i,r}, H_{j,s}] = 0,$}&\\
&\makebox[0pt][l]{$\big[X_{i,0}^{+}, X_{j,0}^{-}\big] = \delta_{i,j} H_{i, 0},$}&\\
&\makebox[0pt][l]{$\big[X_{i,1}^{+}, X_{j,0}^{-}\big] = \delta_{i,j} H_{i, 1} = \big[X_{i,0}^{+}, X_{j,1}^{-}\big],$}&\\
&\makebox[0pt][l]{$\big[H_{i,0}, X_{j,r}^{\pm}\big] = \pm a_{i,j} X_{j,r}^{\pm},$}&\\
&\makebox[0pt][l]{$\big[H_{i, r+1}, X_{j, s}^{\pm}\big] - \big[H_{i, r}, X_{j, s+1}^{\pm}\big] = \pm a_{i,j}\frac{\ve_1+\ve_2}{2} \big\{H_{i, r}, X_{j, s}^{\pm}\big\}-b_{i,j}\frac{\ve_1-\ve_2}{2}\big[H_{i,r},X^\pm_{j,s}\big],$}&\\
&\makebox[0pt][l]{$\big[X_{i, r+1}^{\pm}, X_{j, s}^{\pm}\big] - \big[X_{i, r}^{\pm}, X_{j, s+1}^{\pm}\big] = \pm a_{i,j}\frac{\ve_1+\ve_2}{2} \big\{X_{i, r}^{\pm}, X_{j, s}^{\pm}\big\}-b_{i,j}\frac{\ve_1-\ve_2}{2}\big[X^\pm_{i,r},X^\pm_{j,s}\big],$}&\\
&\sum\limits_{\sigma\in S_{1-a_{i,j}}}\prod_{j=1}^{1-a_{i,j}}\ad\big(X_{i,r_{\sigma(j)}}^{\pm}\big)\big(X_{j,s}^{\pm}\big)= 0 &&\qquad\text{if }i \neq j,&\\
&\big[X^\pm_{i,r},X^\pm_{i,s}\big]=0 &&\qquad \text{if }p(i)\neq p(i+1),&\\
&\big[\big[X^\pm_{i-1,r},X^\pm_{i,0}\big],\big[X^\pm_{i,0},X^\pm_{i+1,s}\big]\big]=0&&\qquad \text{if }p(i)\neq p(i+1),&
\end{alignat*}
where the generators $X^\pm_{i, r}$ are odd if $p(i)\neq p(i+1)$, all other generators are even, $S_l$ is the symmetric group of order $l$ and $\{X,Y\}=XY+YX$.
\end{Definition}
In this article, we use the finite presentation of the affine super Yangian given in~\cite[Proposition~2.23]{U4}.
\begin{Theorem}[{\cite[Proposition~2.23]{U4}}]\label{Prop32}
Let us set $\hbar=\ve_1+\ve_2$, $\ve=-(m-n)\ve_1$. The affine super Yangian \smash{$Y_{\hbar,\ve}\big(\widehat{\mathfrak{sl}}(m|n)\big)$} is the associative superalgebra over $\mathbb{C}$ generated by
\begin{equation*}
\big\{X_{i,r}^\pm, H_{i,r}\mid i\in I_{m|n}=\mathbb{Z}/(m+n)\mathbb{Z},\,r=0,1\big\}
\end{equation*}
subject to the following defining relations:
\begin{gather}
[H_{i,r}, H_{j,s}] = 0,\label{Eq2.1}\\
\bigl[X_{i,0}^{+}, X_{j,0}^{-}\bigr] = \delta_{i,j} H_{i, 0},\label{Eq2.2}\\
\bigl[X_{i,1}^{+}, X_{j,0}^{-}\bigr] = \delta_{i,j} H_{i, 1} = \bigl[X_{i,0}^{+}, X_{j,1}^{-}\bigr],\label{Eq2.3}\\
\bigl[H_{i,0}, X_{j,r}^{\pm}\bigr] = \pm a_{i,j} X_{j,r}^{\pm},\label{Eq2.4}\\
\bigl[\tilde{H}_{i,1}, X_{j,0}^{\pm}\bigr] = \pm a_{i,j} X_{j,1}^{\pm}\hspace{25.2mm} \text{if }(i,j)\neq(0,m+n-1),(m+n-1,0),\label{Eq2.5}\\
\bigl[\widetilde{H}_{0,1}, X_{m+n-1,0}^{\pm}\bigr] = \pm\biggl(X_{m+n-1,1}^{\pm}+\biggl(\ve+\frac{\hbar}{2}(m-n)\hbar\biggr) X_{m+n-1, 0}^{\pm}\biggr),\label{Eq2.6}\\
\bigl[\widetilde{H}_{m+n-1,1}, X_{0,0}^{\pm}\bigr] = \pm\biggl(X_{0,1}^{\pm}-\biggl(\ve+\frac{\hbar}{2}(m-n)\hbar\biggr) X_{0, 0}^{\pm}\biggr),\label{Eq2.7}\\
\bigl[X_{i, 1}^{\pm}, X_{j, 0}^{\pm}\bigr] - \bigl[X_{i, 0}^{\pm}, X_{j, 1}^{\pm}\bigr] = \pm a_{i,j}\frac{\hbar}{2} \bigl\{X_{i, 0}^{\pm}, X_{j, 0}^{\pm}\bigr\}\nonumber \\
\hspace{64.15mm} \text{if }(i,j)\neq(0,m+n-1),(m+n-1,0),\label{Eq2.8}\\
\bigl[X_{0, 1}^{\pm}, X_{m+n-1, 0}^{\pm}\bigr] - \bigl[X_{0, 0}^{\pm}, X_{m+n-1, 1}^{\pm}\bigr]\nonumber \\
\qquad =\pm\frac{\hbar}{2} \bigl\{X_{0, 0}^{\pm}, X_{m+n-1, 0}^{\pm}\bigr\}+\biggl(\ve+\frac{\hbar}{2}(m-n)\hbar\biggr) \bigl[X_{0, 0}^{\pm}, X_{m+n-1, 0}^{\pm}\bigr],\label{Eq2.9}\\
\big(\ad X_{i,0}^{\pm}\big)^{1+|a_{i,j}|} \big(X_{j,0}^{\pm}\big)= 0 \hspace{20.75mm} \text{if }i \neq j, \label{Eq2.10}\\
\bigl[X^\pm_{i,0},X^\pm_{i,0}\bigr]=0 \hspace{38.75mm} \text{if }p(i)\neq p(i+1),\label{Eq2.11}\\
\bigl[\bigl[X^\pm_{i-1,0},X^\pm_{i,0}\bigr],\bigl[X^\pm_{i,0},X^\pm_{i+1,0}\bigr]\bigr]=0 \qquad \text{if }p(i)\neq p(i+1),\label{Eq2.12}
\end{gather}
where the generators \smash{$X^\pm_{i, r}$} are odd if $p(i)\neq p(i+1)$, all other generators are even and we set \smash{$\widetilde{H}_{i,1} = H_{i,1}-\frac{\hbar}2 H_{i,0}^2$} and $\{X,Y\}=XY+YX$.
\end{Theorem}
We note that we set $\ve=-(m-n)\ve_2$ in \cite{U4}.

Let us set an anti-automorphism
\begin{equation*}
\omega\colon\ Y_{\hbar,\ve}\big(\widehat{\mathfrak{sl}}(m|n)\big)\to Y_{\hbar,\ve}\big(\widehat{\mathfrak{sl}}(m|n)\big)
\end{equation*}
given by
\begin{equation*}
\omega(H_{i,r})=H_{i,r},\qquad \omega\big(X^+_{i,r}\big)=(-1)^{p(i)}X^-_{i,r},\qquad \omega\big(X^-_{i,r}\big)=(-1)^{p(i+1)}\big(X^+_{i,r}\big).
\end{equation*}

Let us set a Lie superalgebra
\begin{equation*}
\widehat{\mathfrak{gl}}(m|n)=\mathfrak{gl}(m|n)\otimes\mathbb{C}\big[t^{\pm1}\big]\oplus\mathbb{C}c\oplus\mathbb{C}z
\end{equation*}
with the commutator relations:
\begin{gather*}
[E_{i,j}t^r,E_{x,y}t^s]=\delta_{j,x}E_{i,y}t^{r+s}-(-1)^{p(E_{i,j})p(E_{x,y})}\delta_{i,y}E_{x,j}t^{r+s}\\
\hphantom{[E_{i,j}t^r,E_{x,y}t^s]=}{}
+\delta_{r+s,0}r(-1)^{p(i)}\delta_{i,y}\delta_{j,x}c+\delta_{r+s,0}r(-1)^{p(i)+p(x)}\delta_{i,j}\delta_{x,y}z,\\
\big[c,\widehat{\mathfrak{gl}}(m|n)\big]=\big[z,\widehat{\mathfrak{gl}}(m|n)\big]=0,
\end{gather*}
where $E_{i,j}$ is a matrix unit of $\mathfrak{gl}(m|n)$ whose $(u,v)$ component is $\delta_{i,u}\delta_{j,v}$ and the parity $p(E_{i,j})=p(i)+p(j)$.
We also take a subalgebra
\smash{$\widehat{\mathfrak{sl}}(m|n)=\mathfrak{sl}(m|n)\otimes\mathbb{C}\big[t^{\pm1}\big]\oplus\mathbb{C}c$}. Let us set the Chevalley generators of \smash{$\widehat{\mathfrak{sl}}(m|n)$} as
\begin{gather*}
h_i=\begin{cases}
(-1)^{p(m+n)}E_{m+n,m+n}-E_{1,1}+c&\text{if }i=0,\\
(-1)^{p(i)}E_{i,i}-(-1)^{p(i+1)}E_{i+1,i+1}&\text{if }1\leq i\leq m+n-1,
\end{cases}\\
x^+_i=\begin{cases}
E_{m+n,1}t&\text{if }i=0,\\
E_{i,i+1}&\text{if }1\leq i\leq m+n-1,
\end{cases}\\ x^-_i=\begin{cases}
(-1)^{p(m+n)}E_{1,m+n}t^{-1}&\text{if }i=0,\\
(-1)^{p(i)}E_{i+1,i}&\text{if }1\leq i\leq m+n-1.
\end{cases}
\end{gather*}
According to Definition~\ref{Prop32}, there exists a homomorphism from the universal enveloping algebra \smash{$U\bigl(\widehat{\mathfrak{sl}}(m|n)\bigr)$} to $Y_{\hbar,\ve}\bigl(\widehat{\mathfrak{sl}}(m|n)\bigr)$ given by $h_i\mapsto H_{i,0}$ and \smash{$x^\pm_i\mapsto X^\pm_{i,0}$}. We denote the image of~\smash{$x\in U\bigl(\widehat{\mathfrak{sl}}(m|n)\bigr)$}~via this homomorphism by $x$.
Since $\big[x^\pm_i,x^\pm_j\big]=0$ holds for $|i-j|>1$, we obtain \smash{$\big[X^\pm_{i,0},X^\pm_{j,0}\big]=0$}. By~\eqref{Eq2.5}--\eqref{Eq2.7} and the assumption that $m+n\geq 5$, we have
\begin{align}
&\big[X^\pm_{i,r},X^\pm_{j,s}\big]=0 \qquad \text{if }|i-j|>1.\label{Eq2.13}
\end{align}

Let us set a degree on \smash{$Y_{\hbar,\ve}\bigl(\widehat{\mathfrak{sl}}(m|n)\bigr)$} by
\begin{equation*}
{\rm deg}(H_{i,r})=0,\qquad {\rm deg}\big(X^\pm_{i,r}\big)=\begin{cases}
\pm1&\text{if }i=0,\\
0&\text{if }i\neq 0.
\end{cases}
\end{equation*}
In order to define the edge contraction for the affine super Yangian, we need to use the standard degreewise completion defined in \cite{MNT}. For a $\mathbb{Z}$-graded algebra \smash{$A=\bigoplus_{d\in\mathbb{Z}} A_d$}, we can set a~topology on $A$ as the linear topology defined by the sequence of \smash{$\bigl\{\bigoplus_{d\in\mathbb{Z}}\big(\sum\limits_{r>N}A_{d-r}A_r\big)\bigr\}_N$}. This makes~$A$ a compatible degreewise topological algebra. We take the corresponding degreewise completion of $A$ and call it the standard degreewise completion of $A$.

We denote the standard degreewise completion of \smash{$Y_{\hbar,\ve}\bigl(\widehat{\mathfrak{sl}}(m|n)\bigr)$} by \smash{$\widetilde{Y}_{\hbar,\ve}\bigl(\widehat{\mathfrak{sl}}(m|n)\bigr)$}. Using the same degree as \smash{$\widetilde{Y}_{\hbar,\ve}\bigl(\widehat{\mathfrak{sl}}(m|n)\bigr)$}, we define the standard degreewise completion of \smash{$\otimes^2Y_{\hbar,\ve}\bigl(\widehat{\mathfrak{sl}}(m|n)\bigr)$}
and denote it by \smash{$Y_{\hbar,\ve}\bigl(\widehat{\mathfrak{sl}}(m|n)\bigr)\widehat{\otimes} Y_{\hbar,\ve}\bigl(\widehat{\mathfrak{sl}}(m|n)\bigr)$}.
\begin{Theorem}[{\cite[Theorem~4.3]{U2}}]
There exists an algebra homomorphism
\begin{equation*}
\Delta\colon\ Y_{\hbar,\ve}\bigl(\widehat{\mathfrak{sl}}(m|n)\bigr)\to Y_{\hbar,\ve}\bigl(\widehat{\mathfrak{sl}}(m|n)\bigr)\widehat{\otimes} Y_{\hbar,\ve}\bigl(\widehat{\mathfrak{sl}}(m|n)\bigr)
\end{equation*}
determined by
\begin{gather*}
\Delta\big(X^\pm_{j,0}\big)=X^\pm_{j,0}\otimes1+1\otimes X^\pm_{j,0} \qquad\text{for }0\leq j\leq m+n-1,\\
\Delta\big(X^+_{i,1}\big)=X^+_{i,1}\otimes1+1\otimes X^+_{i,1}+B_i \qquad\text{for }1\leq i\leq m+n-1,
\end{gather*}
where we set $B_i$ as
\begin{gather*}
B_i =\hbar\sum\limits_{s \geq 0}  \sum\limits_{u=1}^{i}  \big((-1)^{p(u)}E_{i,u}t^{-s}\otimes E_{u,i+1}t^s \\
\hphantom{B_i =\hbar\sum\limits_{s \geq 0}  \sum\limits_{u=1}^{i}  \bigl(}{}\
-(-1)^{p(u)+p(E_{i,u})p(E_{i+1,u})}E_{u,i+1}t^{-s-1}\otimes E_{i,u}t^{s+1}\big)\\
\hphantom{B_i =}{}
+\hbar\sum\limits_{s \geq 0}  \sum\limits_{u=i+1}^{m+n}  \bigl((-1)^{p(u)}E_{i,u}t^{-s-1}\otimes E_{u,i+1}t^{s+1}\\
\hphantom{B_i =+\hbar\sum\limits_{s \geq 0}  \sum\limits_{u=i+1}^{m+n}  \bigl(}{}\
-(-1)^{p(u)+p(E_{i,u})p(E_{i+1,u})}E_{u,i+1}t^{-s}\otimes E_{i,u}t^{s}\bigr).
\end{gather*}
\end{Theorem}
Since $\Delta$ satisfies the coassociativity, $\Delta$ can be considered as the coproduct for the affine super Yangian.
\section{Edge contractions for the affine super Yangian}
In \cite{U13}, we gave two edge contractions for the affine super Yangian. In the following theorem, we do not identify $I_{m|n}$ with $\mathbb{Z}/(m+n)\mathbb{Z}$.
\begin{Theorem}[{\cite[Sections 6--9 and Theorem~11.1]{U13}}]\qquad
\begin{enumerate}
\item[$1.$] For $m_2,n_2\geq 0$, $m_1,n_1\geq 2$ and $m_1+n_1\geq 5$, there exists a homomorphism
\begin{gather*}
\Psi_1^{m_1|n_1,m_1+m_2|n_1+n_2}\colon\ Y_{\hbar,\ve}\big(\widehat{\mathfrak{sl}}(m_1|n_1)\big)\to \widetilde{Y}_{\hbar,\ve}\big(\widehat{\mathfrak{sl}}(m_1+m_2|n_1+n_2)\big)
\end{gather*}
given by
\begin{gather*}
\Psi_1^{m_1|n_1,m_1+m_2|n_1+n_2}\big(X^+_{i,0}\big)=\begin{cases}
E_{-n_1,1}t&\text{if }i=-n_1,\\
E_{i,i+1}&\text{if }1\leq i\leq m_1-1,\\
E_{m_1,-1}&\text{if }i=m_1,\\
E_{i,i-1}&\text{if }-n_1+1\leq i\leq -1,
\end{cases}\\
\Psi_1^{m_1|n_1,m_1+m_2|n_1+n_2}\big(X^-_{i,0}\big)=\begin{cases}
-E_{1,-n_1}t^{-1}&\text{if }i=n_1,\\
E_{i+1,i}&\text{if }1\leq i\leq m_1-1,\\
E_{-1,m_1}&\text{if }i=m_1,\\
-E_{i-1,i}&\text{if }-n_1+1\leq i\leq -1
\end{cases}
\end{gather*}
and
\begin{align*}
&\Psi_1^{m_1|n_1,m_1+m_2|n_1+n_2}\big(\widetilde{H}_{1,1}\big)
= \widetilde{H}_{1,1}-P_1+P_2+Q_1-Q_2,\\
&\Psi_1^{m_1|n_1,m_1+m_2|n_1+n_2}\big(X^+_{1,1}\big)
= X^+_{1,1}-P^+_1+Q^+_1,
\end{align*}
where
\begin{gather*}
P_i=\hbar\sum\limits_{v\geq0}  \sum\limits_{z=m_1+1}^{m_1+m_2}  E_{i,z}t^{-v-1} E_{z,i}t^{v+1},\\
Q_i=\hbar\sum\limits_{v\geq0}  \sum\limits_{z=-n_1-n_2}^{-n_1-1}  E_{i,z}t^{-v-1} E_{z,i}t^{v+1},\\
P^+_i\!=\hbar\sum\limits_{v\geq0}  \sum\limits_{z=m_1+1}^{m_1+m_2}  E_{i,z}t^{-v-1} E_{z,i+1}t^{v+1},\\
Q^+_i\!=\hbar\sum\limits_{v\geq0}  \sum\limits_{z=-n_1-n_2}^{-n_1-1}  E_{i,z}t^{-v-1} E_{z,i+1}t^{v+1},
\end{gather*}
\smash{$\widetilde{Y}_{\hbar,\ve}\big(\widehat{\mathfrak{sl}}(m_1+m_2|n_1+n_2)\big)$} is the standard degreewise completion of $Y_{\hbar,\ve}\big(\widehat{\mathfrak{sl}}(m_1+m_2|n_1+n_2)\big)$.
\item[$2.$] For $m_1,n_1\geq 0$, $m_2,n_2\geq 2$ and $m_2+n_2\geq5$, there exists a homomorphism
\begin{gather*}
\Psi_2^{m_2|n_2,m_1+m_2|n_1+n_2}\colon\ Y_{\hbar,\ve+(m_1-n_1)\hbar}\big(\widehat{\mathfrak{sl}}(m_2|n_2)\big)\to \widetilde{Y}_{\hbar,\ve}\big(\widehat{\mathfrak{sl}}(m_1+m_2|n_1+n_2)\big)
\end{gather*}
determined by
\begin{gather*}
\Psi_2^{m_2|n_2,m_1+m_2|n_1+n_2}\big(X^+_{i,0}\big)=\begin{cases}
E_{-n_1-n_2,m_1+1}t&\text{if }i=-n_2,\\
E_{m_1+i,m_1+i+1}&\text{if }1\leq i\leq m_2-1,\\
E_{m_1+m_2,-n_1-1}&\text{if }i=m_2,\\
E_{-n_1+i,-n_1+i-1}&\text{if }-n_2+1\leq i\leq -1,
\end{cases}\\
\Psi_2^{m_2|n_2,m_1+m_2|n_1+n_2}\big(X^-_{i,0}\big)=\begin{cases}
-E_{m_1+1,-n_1-n_2}t^{-1}&\text{if }i=-n_2,\\
E_{m_1+i+1,m_1+i}&\text{if }1\leq i\leq m_2-1,\\
E_{-n_1-1,m_1+m_2}&\text{if }i=m_2,\\
-E_{-n_1+i-1,-n_1+i}&\text{if }-n_2+1\leq i\leq -1,
\end{cases}
\end{gather*}
and
\begin{align*}
&\Psi_2^{m_2|n_2,m_1+m_2|n_1+n_2}\big(\widetilde{H}_{1,1}\big)
= \widetilde{H}_{1+m_1,1}+R_1-R_2+S_1-S_2,\\
&\Psi_2^{m_2|n_2,m_1+m_2|n_1+n_2}\big(X^+_{1,1}\big)
= X^+_{1+m_1,1}+R_1^++S_1^+,
\end{align*}
where
\begin{gather*}
R_i=\hbar\sum\limits_{v\geq0} \sum\limits_{z=-n_1}^{-1} E_{z,i+m_1}t^{-v}E_{i+m_1,z}t^{v},\\
S_i=\hbar\sum\limits_{v\geq0} \sum\limits_{z=1}^{m_1} E_{z,i+m_1}t^{-v-1} E_{i+m_1,z}t^{v+1},\\
R^+_i=\hbar\sum\limits_{v\geq0} \sum\limits_{z=-n_1}^{-1} E_{z,i+1+m_1}t^{-v}E_{i+m_1,z}t^{v},\\
S^+_i=\hbar\sum\limits_{v\geq0} \sum\limits_{z=1}^{m_1} E_{z,i+1+m_1}t^{-v-1} E_{i+m_1,z}t^{v+1}.
\end{gather*}
\end{enumerate}
\end{Theorem}
Similarly to \cite[Theorem~4.2]{U10}, we obtain the following theorem.
\begin{Theorem}\label{Commutativity}
The images of \smash{$\Psi_1^{m_1|n_1,m_1+m_2|n_1+n_2}$} and \smash{$\Psi_2^{m_2|n_2,m_1+m_2|n_1+n_2}$} commute with each other.
\end{Theorem}

\section[Affine super Yangians and centralizer algebras of U(hat{gl}(n))]{Affine super Yangians and centralizer algebras of $\boldsymbol{U\bigl(\widehat{\mathfrak{gl}}(n)\bigr)}$}

Following \cite{MNT}, we consider a completion of \smash{$U\bigl(\widehat{\mathfrak{gl}}(m|n)\bigr)/U\bigl(\widehat{\mathfrak{gl}}(m|n)\bigr)(z-1)$}, which is a quotient algebra of \smash{$U\bigl(\widehat{\mathfrak{gl}}(m|n)\bigr)$} divided by the relation $z-1$.
We take the grading of
\[
U\bigl(\widehat{\mathfrak{gl}}(m|n)\bigr)/U\bigl(\widehat{\mathfrak{gl}}(m|n)\bigr)(z-1)
\]
as ${\rm deg}(Xt^s)=s$ and ${\rm deg}(c)=0$. We denote by \smash{$\mathcal{U}\bigl(\widehat{\mathfrak{gl}}(m|n)\bigr)$} the standard degreewise completion of~\smash{$U\bigl(\widehat{\mathfrak{gl}}(m|n)\bigr)/U\bigl(\widehat{\mathfrak{gl}}(m|n)\bigr)(z-1)$}.
\begin{Theorem}[{\cite[Theorem~5.1]{U2}} and {\cite[Theorem~3.29]{U3}}]\label{thm:main}\qquad
\begin{enumerate}
\item[$(1)$] Let $\hat{i}$ be $\sum\limits_{u=1}^i(-1)^{p(u)}$ for $1\leq i\leq m+n-1$. Suppose that $\hbar\neq0$ and $c=\frac{\ve}{\hbar}$.
For a~complex number $a$, there exists an algebra homomorphism
\begin{equation*}
\ev_{\hbar,\ve}^{m|n,a} \colon\ Y_{\hbar,\ve}\bigl(\widehat{\mathfrak{sl}}(m|n)\bigr) \to \mathcal{U}\bigl(\widehat{\mathfrak{gl}}(m|n)\bigr)
\end{equation*}
uniquely determined by
\begin{gather*}
\ev_{\hbar,\ve}^{m|n,a}\big(X_{i,0}^{+}\big) = \begin{cases}
E_{m+n,1}t&\text{if }i=0,\\
E_{i,i+1}&\text{if }1\leq i\leq m+n-1,
\end{cases} \\
\ev_{\hbar,\ve}^{m|n,a}\big(X_{i,0}^{-}\big) = \begin{cases}
(-1)^{p(m+n)}E_{1,m+n}t^{-1}&\text{if }i=0,\\
(-1)^{p(i)}E_{i+1,i}&\text{if }1\leq i\leq m+n-1,
\end{cases}
\end{gather*}
and
\begin{align*}
\ev_{\hbar,\ve}^{m|n,a}\big(X^+_{i,1}\big)&=\biggl(a-\frac{\hat{i}}{2}\hbar\biggr) E_{i,i+1}+ \hbar \sum\limits_{s \geq 0} \sum\limits_{u=1}^{i}  (-1)^{p(u)}E_{i,u}t^{-s}E_{u,i+1}t^s\\
&\quad+\hbar \sum\limits_{s \geq 0}  \sum\limits_{u=i+1}^{m+n}  (-1)^{p(u)}E_{i,u}t^{-s-1}E_{u,i+1}t^{s+1}\qquad \text{for }i\neq0.
\end{align*}
\item[$(2)$] In the case that $\ve\neq 0$, the image of the evaluation map is dense in $\mathcal{U}\bigl(\widehat{\mathfrak{gl}}(m|n)\bigr)$.
\end{enumerate}
\end{Theorem}
Let us set $\iota_1$ as an embedding from $\mathcal{U}\bigl(\widehat{\mathfrak{gl}}(m_1|n_1)\bigr)$ to $\mathcal{U}\bigl(\widehat{\mathfrak{gl}}(m_1+m_2|n_1+n_2)\bigr)$ by
\begin{equation*}
E_{i,j}t^s\mapsto E_{i,j}t^s,\qquad c\mapsto c
\end{equation*}
for $i,j\in I$. For an associative superalgebra $A$ and its subalgebra $B$, we also define the centralizer~algebra
\begin{equation*}
C(A,B)=\{x\in A\mid[x,B]=0\}.
\end{equation*}
\begin{Theorem}\quad
\begin{enumerate}
\item[$(1)$] Let us assume that $c=\frac{\ve}{\hbar}$. The following relation holds:
\begin{equation*}
\iota_1\circ\ev_{\hbar,\ve}^{m_1|n_1,a}=\ev_{\hbar,\ve}^{m_1+m_2|n_1+n_2,a}\circ\Psi_1^{m_1|n_1,m_1+m_2|n_1+n_2}
\end{equation*}

\item[$(2)$] In the case that $c=\frac{\ve}{\hbar}$ and $c\neq0$, the image of \smash{$\ev_{\hbar,\ve}^{m_1+m_2|n_1+n_2,a}\circ\Psi_2^{m_2|n_2,m_1+m_2|n_1+n_2}$} is~contained in \smash{$C\big(\mathcal{U}\bigl(\widehat{\mathfrak{gl}}(m_1+m_2|n_1+n_2)\big),\mathcal{U} \big(\widehat{\mathfrak{gl}}(m_1|n_1)\big)\big)$}.
\end{enumerate}
\end{Theorem}
\begin{proof}
(1) Since the affine super Yangian $Y_{\hbar,\ve}\big(\widehat{\mathfrak{sl}}(m_1|n_1)\big)$ is generated by \smash{$\big\{X^\pm_{i,0}\big\}_{i\in I_{m_1|n_1}}$} and $X^+_{1,1}$, it is enough to show the following relations:
\begin{gather*}
\iota_1\circ\ev_{\hbar,\ve}^{m_1|n_1,a}\big(X^\pm_{i,0}\big)=\ev_{\hbar,\ve}^{m_1+m_2|n_1+n_2,a}\circ\Psi_1^{m_1|n_1,m_1+m_2|n_1+n_2}\big(X^\pm_{i,0}\big),\\
\iota_1\circ\ev_{\hbar,\ve}^{m_1|n_1,a}\big(X^+_{1,1}\big)=\ev_{\hbar,\ve}^{m_1+m_2|n_1+n_2,a}\circ\Psi_1^{m_1|n_1,m_1+m_2|n_1+n_2}\big(X^+_{1,1}\big).
\end{gather*}
These relations follow from the definition of \smash{$\ev_{\hbar,\ve}^{m_1|n_1,a}$} and \smash{$\Psi_1^{m_1|n_1,m_1+m_2|n_1+n_2}$}.

(2) By Theorem~\ref{Commutativity}, the image of
\[
\ev_{\hbar,\ve}^{m_1+m_2|n_1+n_2,a}\circ\Psi_2^{m_2|n_2,m_1+m_2|n_1+n_2}
\]
is commutative with the one of
\[
\ev_{\hbar,\ve}^{m_1+m_2|n_1+n_2,a}\circ\Psi_1^{m_1|n_1,m_1+m_2|n_1+n_2}.
\]
By item~(1) and Theorem~\ref{thm:main}\,(2), the completion~of the image of
\[
\ev_{\hbar,\ve}^{m_1+m_2|n_1+n_2,a}\circ\Psi_1^{m_1|n_1,m_1+m_2|n_1+n_2}
\]
 coincides with \smash{$\iota_1\big(\mathcal{U}\big(\widehat{\mathfrak{gl}}(m_1|n_1)\big)\big)$}. Thus, the image of
 \[
 \ev_{\hbar,\ve}^{m_1+m_2|n_1+n_2,a}\circ\Psi_2^{m_2|n_2,m_1+m_2|n_1+n_2}
  \]
  is contained in the centralizer algebra $C\big(\mathcal{U}\big(\smash{\widehat{\mathfrak{gl}}}(m_1+m_2|n_1+n_2)\big), \mathcal{U}\big(\smash{\widehat{\mathfrak{gl}}}(m_1|n_1)\big)\big)$.
\end{proof}

\section[W-superalgebras of type A]{$\boldsymbol{W}$-superalgebras of type $\boldsymbol{A}$}
Let us set some notations of a vertex superalgebra. For a vertex superalgebra $V$, we denote the generating field associated with $v\in V$ by \smash{$v(z)=\sum\limits_{s\in\mathbb{Z}} v_{(s)}z^{-s-1}$}. We also denote the operator product expansion (OPE) of $V$ by
\begin{equation*}
u(z)v(w)\sim\sum\limits_{s\geq0}  \frac{(u_{(s)}v)(w)}{(z-w)^{s+1}}
\end{equation*}
for all $u, v\in V$. We denote the vacuum vector (resp.\ the translation operator) by $|0\rangle$ (resp.\ $\partial$).

We denote the universal affine vertex superalgebra associated with a finite dimensional Lie superalgebra $\mathfrak{g}$ and its inner product $\kappa$ by $V^\kappa(\mathfrak{g})$. By the PBW theorem, we can identify $V^\kappa(\mathfrak{g})$ with $U\big(t^{-1}\mathfrak{g}\big[t^{-1}\big]\big)$. In order to simplify the notation, here after, we denote the generating field~$\big(ut^{-1}\big)(z)$ as $u(z)$ for $u\in\mathfrak{g}$. By the definition of $V^\kappa(\mathfrak{g})$, the generating fields $u(z)$ and~$v(z)$ satisfy the OPE
\begin{gather*}
u(z)v(w)\sim\frac{[u,v](w)}{z-w}+\frac{\kappa(u,v)}{(z-w)^2}\label{OPE1}
\end{gather*}
for all $u,v\in\mathfrak{g}$.

We take two positive integers and their partitions:
\begin{align}
M&=\sum\limits_{i=1}^lu_i,\qquad u_1\geq u_{2}\geq\dots\geq u_l\geq0,\nonumber\\
N&=\sum\limits_{i=1}^lq_i,\qquad q_1\geq q_{2}\geq\dots\geq q_l\geq0,\label{cond:q}
\end{align}
satisfying that $M\neq N$ and $u_l+q_l\neq0$.
For $1\leq i\leq M$ and $-N\leq j\leq -1$, we set ${1\leq \col(i),\col(j)\leq l}$, $u_1-u_{\col(i)}<\row(i)\leq u_1$ and $-q_1\leq\row(j)<-q_1+q_{\col(j)}$ satisfying
\begin{gather*}
\sum\limits_{b=1}^{\col(i)-1}u_b<i\leq\sum\limits_{b=1}^{\col(i)}u_b,\qquad
\row(i)=i-\sum\limits_{b=1}^{\col(i)-1}u_b+u_1-u_{\col(i)},\\
\sum\limits_{b=1}^{\col(j)-1}q_b<-j\leq\sum\limits_{b=1}^{\col(j)}q_b,\qquad \row(j)=j+\sum\limits_{b=1}^{\col(j)-1}q_b-q_1+q_{\col(j)}.
\end{gather*}
The definition of $\col$ and $\row$ can be interpreted by using the Young diagram. For the partition $(u_1,u_2,\dots,u_l)$ (resp.\ $(q_1,q_2,\dots,q_l)$), we define $D_M$ (resp.\ $D_N$) as the Young diagram in French style corresponding to this partition. We enumerate boxes in $D_M$ (resp.\ $D_N$) by $1,2,\dots,M$ (resp.\ $-1,-2,\dots,-N$) down columns from left to right. Then, $\col(i)$ denotes the column in which the number $i$ is located, while $\row(i)$ denotes the column number of the number $i$ from the top.

Let us set a Lie superalgebra \smash{$\mathfrak{gl}(M|N)=\bigoplus_{i,j\in I_{M|N}}\mathbb{C}e_{i,j}$} whose commutator relations are determined by
\begin{equation*}
[e_{i,j},e_{x,y}]=\delta_{j,x}e_{i,y}-(-1)^{p(e_{i,j})p(e_{x,y})}\delta_{i,y}e_{x,j},
\end{equation*}
where $p(e_{i,j})=p(i)+p(j)$.
We take a nilpotent element $f\in\mathfrak{gl}(M|N)$ as
\begin{equation*}
f=\sum\limits_{i\in I_{M|N}}  e_{\hat{i},i},
\end{equation*}
where the integer $\hat{i}\in I_{M|N}$ are determined by
$\col(\hat{i})=\col(i)+1$, $\row(\hat{i})=\row(i)$.

\begin{Remark}
Actually, if the nilpotent element has a good grading (see \cite[Theorem~7.2]{Ho}), the discussion after here works well (see \cite{U11} and \cite{U13}). For the simplicity, we assume the condition~\eqref{cond:q}.
\end{Remark}
Similarly to \smash{$\hat{i}$}, we set \smash{$\tilde{i}\in I_{M|N}$} as
$\col(\tilde{i})=\col(i)-1$, $\row(\tilde{i})=\row(i)$.
We also fix an inner product of the Lie superalgebra $\mathfrak{gl}(M|N)$ determined by
\begin{equation*}
(e_{i,j}|e_{x,y})=k\delta_{i,y}\delta_{x,j}(-1)^{p(i)}+\delta_{i,j}\delta_{x,y}(-1)^{p(i)+p(x)}.
\end{equation*}
We set two Lie superalgebra
\begin{equation*}
\mathfrak{b}=\bigoplus_{\substack{i,j\in I_{M|N},\\\col(i)\geq\col(j)}}\mathbb{C}e_{i,j},\qquad
\mathfrak{a}=\mathfrak{b}\oplus\bigoplus_{\substack{i,j\in I_{M|N},\\\col(i)>\col(j)}}\mathbb{C}\psi_{i,j},
\end{equation*}
whose commutator relations are defined by
\begin{gather*}
[e_{i,j},\psi_{x,y}]=\delta_{j,x}\psi_{i,y}-\delta_{i,y}(-1)^{p(e_{i,j})(p(e_{x,y})+1)}\psi_{x,j},\\
[\psi_{i,j},\psi_{x,y}]=\delta_{j,x}\psi_{i,y}-\delta_{i,y}(-1)^{(p(e_{i,j})+1)(p(e_{x,y})+1)}\psi_{x,j},
\end{gather*}
where the parity of $e_{i,j}$ is $p(i)+p(j)$ and the parity of $\psi_{i,j}$ is $p(i)+p(j)+1$.
We also set an inner product on $\mathfrak{b}$ and $\mathfrak{a}$ by
\begin{equation*}
\kappa(e_{i,j},e_{p,q})=(e_{i,j}|e_{p,q}),\qquad \kappa(e_{i,j},\psi_{p,q})=\kappa(\psi_{i,j},\psi_{p,q})=0.
\end{equation*}
We denote the universal affine vertex superalgebras associated with $\mathfrak{b}$ and $\mathfrak{a}$ by $V^\kappa(\mathfrak{b})$ and $V^\kappa(\mathfrak{a})$. We also sometimes denote the elements $e_{i,j}t^{-s}\in V^\kappa(\mathfrak{b})\subset V^\kappa(\mathfrak{a})$ and $\psi_{i,j}t^{-s}\in V^\kappa(\mathfrak{a})$ by $e_{i,j}[-s]$ and $\psi_{i,j}[-s]$ respectively and $a_{(-1)}b$ by $ab$.
Let us define an odd differential $d_0 \colon V^{\kappa}(\mathfrak{b})\to V^{\kappa}(\mathfrak{a})$ determined by
\begin{gather*}
d_01=0,\\
[d_0,\partial]=0,
\\
d_0(e_{i,j}[-1]])
=\sum\limits_{\substack{\col(i)>\col(r)\geq\col(j)}}  (-1)^{p(e_{i,j})+p(e_{i,r})p(e_{r,j})}e_{r,j}[-1]\psi_{i,r}[-1]\nonumber\\
\hphantom{d_0(e_{i,j}[-1]])=}{}-\sum\limits_{\substack{\col(j)<\col(r)\leq\col(i)}}  (-1)^{p(e_{i,r})p(e_{r,j})}\psi_{r,j}[-1]e_{i,r}[-1]+\delta(\col(i)\nonumber\\
\hphantom{d_0(e_{i,j}[-1]])}{}>\col(j))(-1)^{p(i)}\alpha_{\col(i)}\psi_{i,j}[-2]+(-1)^{p(i)}\psi_{\hat{i},j}[-1]-(-1)^{p(i)}\psi_{i,\tilde{j}}[-1].
\end{gather*}
By using \cite[Theorem~2.4]{KRW}, we can define the $W$-algebra $\mathcal{W}^k(\mathfrak{gl}(M|N),f)$ as follows.
\begin{Definition}
The $W$-algebra $\mathcal{W}^k(\mathfrak{gl}(M|N),f)$ is the vertex subalgebra of $V^\kappa(\mathfrak{b})$ defined by
\begin{equation*}
\mathcal{W}^k(\mathfrak{gl}(M|N),f)=\{y\in V^\kappa(\mathfrak{b})\mid d_0(y)=0\}.
\end{equation*}
In the case that $u_1=u_2=\dots=u_l$, $q_1=q_2=\dots=q_l$, we call $\mathcal{W}^k(\mathfrak{gl}(M|N),f)$ the rectangular $W$-superalgebra of type $A$ and denote it by $\mathcal{W}^k\big(\mathfrak{gl}(ml|nl),\big(l^{m|n}\big)\big)$.
\end{Definition}
We give one example. In the case $l=2$, we can write $f$ and $d_0$ as
\begin{align*}
f&=\sum\limits_{1\leq z\leq u_2}  e_{z+u_1,z+u_1-u_2}+\sum\limits_{1\leq z\leq q_2}  e_{-z-q_1,-z-q_1+q_2}
\end{align*}
and
\begin{gather*}
d_0(e_{i,j}[-1])=\delta(\col(i)=1)e_{\hat{i},j}[-1]-\delta(\col(i)=2)e_{i,\tilde{j}}[-1]\qquad \text{if }\col(i)=\col(j),\\
d_0(e_{i,j}[-1])=\sum\limits_{r=1}^{u_1}  (-1)^{p(e_{i,j})+p(i)p(j)}e_{r,j}[-1]\psi_{i,r}[-1]-\sum\limits_{r=-q_1}^{-1}  (-1)^{p(i)p(j)}e_{r,j}[-1]\psi_{i,r}[-1]\nonumber\\
\hphantom{d_0(e_{i,j}[-1])=}{}-\sum\limits_{r=u_1+1}^{u_1+u_2}  (-1)^{p(i)p(j)}\psi_{r,j}[-1]e_{i,r}[-1]\\
\hphantom{d_0(e_{i,j}[-1])=}{}-\sum\limits_{r=-q_1-q_2}^{-q_1-1}  (-1)^{(p(i)+1)(p(j)+1)}\psi_{r,j}[-1]e_{i,r}[-1]\nonumber\\
\hphantom{d_0(e_{i,j}[-1])=}{}+(-1)^{p(i)}\alpha_2\psi_{i,j}[-2] \qquad\text{if }\col(i)=2,\col(j)=1,
\end{gather*}
where
\begin{gather*}
\hat{i}=\begin{cases}
i+u_2&\text{if }u_1-u_2+1\leq i\leq u_1,\\
i-q_2&\text{if }-q_1\leq i\leq-q_1+q_2-1,
\end{cases}\qquad \text{and} \\
\tilde{i}=\begin{cases}
i-u_2&\text{if }u_1+1\leq i\leq u_1+u_2,\\
i+q_2&\text{if }-q_1-q_2\leq i\leq -q_1-1.
\end{cases}
\end{gather*}
We define the set
\begin{equation*}
I_s=\{1,\dots,u_s,-1,\dots,-q_s\}.
\end{equation*}
We constructed two kinds of elements \smash{$W^{(1)}_{a,b},W^{(2)}_{a,b}\in\mathcal{W}^k(\mathfrak{gl}(M|N),f)$} for $a,b\in I_s\setminus I_{s+1}$. Let us~set
\begin{gather*}
\alpha_s=k+M-N-u_s+q_s,\qquad \gamma_a=\sum\limits_{s=a+1}^{l}  \alpha_{s}.
\end{gather*}
and denote $e_{i,j}$ by \smash{$e^{(r)}_{a,b}$} if $\col(i)=\col(j)=r$, $\row(i)=a$, $\row(j)=b$.
\begin{Theorem}[{\cite[Theorem~10.22]{U13}}]
The following elements of $\bigotimes_{1\leq s\leq l} V^{\kappa_s}(\mathfrak{gl}(q_s))$
\begin{equation*}
\bigl\{W^{(1)}_{a,b},W^{(2)}_{a,b}\mid a,b\in I_s\setminus I_{s-1}\bigr\}
\end{equation*}
 are contained in \smash{$\mu\big(\mathcal{W}^k(\mathfrak{gl}(M|N),f)\big)$}:
\begin{gather*}
W^{(1)}_{a,b}=\sum\limits_{1\leq r\leq s}e^{(r)}_{a,b}[-1],\\
W^{(2)}_{a,b}=\sum\limits_{\substack{\col(i)=\col(j)+1\\\row(i)=a,\row(j)=b}}e_{i,j}-\sum\limits_{1\leq r\leq s}\gamma_{r}e^{(r)}_{a,b}[-2]\\
\hphantom{W^{(2)}_{a,b}=}{}+\sum\limits_{\substack{1\leq r_1<r_2\leq s\\x> u_1-u_s}}  (-1)^{p(x)+p(e_{i,v})p(e_{x,j})}e^{(r_1)}_{x,b}[-1]e^{(r_2)}_{a,x}[-1]\\
\hphantom{W^{(2)}_{a,b}=}{}+\sum\limits_{\substack{1\leq r_1<r_2\leq s\\x< -q_1+q_s}}  (-1)^{p(x)+p(e_{a,x})p(e_{b,x})}e^{(r_1)}_{x,b}[-1]e^{(r_2)}_{a,x}[-1]\\
\hphantom{W^{(2)}_{a,b}=}{}-\sum\limits_{\substack{r_1\geq r_2\\q_s-q_1\leq x\leq q_{r_1}-q_1\\\row(i)=a,\row(j)=b}}  (-1)^{p(x)+p(e_{a,x})p(e_{x,b})}e^{(r_1)}_{x,b}[-1]e^{(r_2)}_{a,x}[-1]\\
\hphantom{W^{(2)}_{a,b}=}{}-\sum\limits_{\substack{r_1\geq r_2\\u_1-u_{r_1}\leq x\leq u_1-u_s}}  (-1)^{p(x)+p(e_{a,x})p(e_{x,b})}e^{(r_1)}_{x,b}[-1]e^{(r_2)}_{a,x}[-1].
\end{gather*}
\end{Theorem}
In the rectangular case, we have computed the OPEs in \cite[Section~4]{U4}.
\begin{Corollary}\label{Cor}\quad
\begin{enumerate}
\item[$(1)$] Assume that $u_1<u_2$, $q_1<q_2$. Then, we have an embedding
\begin{equation*}
\iota_2\colon\ \mathcal{W}^{\widehat{k}}\big(\mathfrak{gl}(2u_1|2q_1),\big(2^{u_1|q_1}\big)\big)\to\mathcal{W}^k\big(\mathfrak{gl}(2u_2|2q_2),\big(2^{u_2|q_2}\big)\big),\qquad W^{(r)}_{i,j}\mapsto W^{(r)}_{i,j},
\end{equation*}
where $\widehat{k}=k+u_2-q_2-u_1+q_1$.

\item[$(2)$] Assume that $u_1-u_2,q_1-q_2>0$ and
\begin{equation*}
\kappa_1(E_{i,j},E_{x,y})=\delta_{i,y}\delta_{j,x}(-1)^{p(i)}\alpha_1+(-1)^{p(i)+p(x)}\delta_{i,j}\delta_{x,y}.
\end{equation*}
Then, we can define an embedding
\begin{equation*}
\iota_3\colon\ V^{\kappa_1}(\mathfrak{gl}(u_1-u_2|q_1-q_2))\to\mathcal{W}^k(\mathfrak{gl}(M|N),f),\qquad E_{i,j}[-s]\mapsto W^{(1)}_{i,j}[-s].
\end{equation*}

\item[$(3)$] Assume that $u_1=u_2>u_3$, $q_1-q_2>0$. Then, we have an embedding
\begin{gather*}
\iota_4\colon\ \mathcal{W}^{\widetilde{k}}\big(\mathfrak{gl}(2(u_1-u_3)|2(q_1-q_3)),\big(2^{u_1-u_3|q_1-q_3}\big)\big)\to\mathcal{W}^k(\mathfrak{gl}(M|N),f),\\
\hphantom{\iota_4\colon\ }{} \quad
W^{(r)}_{i,j}\mapsto W^{(r)}_{i,j},
\end{gather*}
where $\widetilde{k}=k+M-N-2(u_1-q_1)+(u_3-q_3)$.
\end{enumerate}
\end{Corollary}
\begin{proof}
(1) follows directly from \cite[Section 4]{U4}. (2) follows from \cite[Section~4]{U4} and the definition of \smash{$W^{(1)}_{i,j}$}. Since the form of \smash{$W^{(r)}_{i,j}\in\mathcal{W}^{\widetilde{k}}\big(\mathfrak{gl}(2(u_1-u_3)|2(q_1-q_3)),\big(2^{u_1-u_3|q_1-q_3}\big)\big)$} and \smash{${W^{(r)}_{i,j}\in\mathcal{W}^k(\mathfrak{gl}(M|N),f)}$} are same, we obtain (3).
\end{proof}

\section[Affine super Yangians and W-superalgebras of type A]{Affine super Yangians and $\boldsymbol{W}$-superalgebras of type $\boldsymbol{A}$}

Let us recall the definition of the universal enveloping algebras of vertex superalgebras.
For any vertex superalgebra $V$, let $L(V)$ be the Bouchard's Lie algebra, that is,
\begin{align*}
 L(V)=V{\otimes}\mathbb{C}\big[t,t^{-1}\big]/\text{Im}\biggl(\partial\otimes\id +\id\otimes\frac{{\rm d}}{{\rm d} t}\biggr)\label{844},
\end{align*}
where the commutation relation is given by
\begin{align*}
 \big[ut^a,vt^b\big]=\sum\limits_{r\geq 0}\begin{pmatrix} a\\r\end{pmatrix}\big(u_{(r)}v\big)t^{a+b-r}
\end{align*}
for all $u,v\in V$ and $a,b\in \mathbb{Z}$.
\begin{Definition}[Frenkel--Zhu~\cite{FZ}, Matsuo--Nagatomo--Tsuchiya~\cite{MNT}]
We set $\mathcal{U}(V)$ as the quotient algebra of the standard degreewise completion of the universal enveloping algebra of $L(V)$ by the completion of the two-sided ideal generated by
\begin{gather*}
\big(u_{(a)}v\big)t^b-\sum\limits_{i\geq 0}
\begin{pmatrix}
 a\\i
\end{pmatrix}
(-1)^i\big(ut^{a-i}vt^{b+i}-{(-1)}^{p(u)p(v)}(-1)^avt^{a+b-i}ut^{i}\big),\\
|0\rangle t^{-1}-1,
\end{gather*}
where $|0\rangle$ is the identity vector of $V$.
We call $\mathcal{U}(V)$ the universal enveloping algebra of $V$.
\end{Definition}
By the definition of the universal affine vertex algebra $V^{\kappa}(\mathfrak{g})$ associated with a finite dimensional reductive Lie superalgebra $\mathfrak{g}$ and the inner product $\kappa$ on $\mathfrak{g}$, $\mathcal{U}(V^\kappa(\mathfrak{g}))$ is the standard degreewise completion of the universal enveloping algebra of the affinization of $\mathfrak{g}$.

Let us set $\kappa_s$ as an inner product on $\mathfrak{gl}(u_s|q_s)$ given by
\begin{equation*}
\kappa_s(e_{i,j},e_{x,y})=(-1)^{p(i)}\alpha_s\delta_{i,y}\delta_{j,x}+(-1)^{p(i)+p(x)}\delta_{i,j}\delta_{x,y}.
\end{equation*}
By \cite[Theorem~5.2]{Genra} and \cite[Theorem~14]{Nak}, there exists an embedding
\begin{equation*}
\mu\colon\ \mathcal{W}^k(\mathfrak{gl}(M|N),f)\to\bigotimes_{1\leq s\leq l}V^{\kappa_s}(\mathfrak{gl}(u_s|q_s)).
\end{equation*}
This embedding is called the Miura map.
Then, induced by the Miura map $\mu$, we obtain the embedding
\begin{equation*}
\widetilde{\mu}\colon\ \mathcal{U}\big(\mathcal{W}^{k}(\mathfrak{gl}(M|N),f)\big)\to {\widehat{\bigotimes}}_{1\leq a\leq l}U\big(\widehat{\mathfrak{gl}}(u_a|q_a)\big),
\end{equation*}
where \smash{${\widehat{\bigotimes}}_{1\leq a\leq l}U\big(\widehat{\mathfrak{gl}}(u_a|q_a)\big)$} is the standard degreewise completion of \smash{$\bigotimes_{1\leq a\leq l}U\big(\widehat{\mathfrak{gl}}(u_a|q_a)\big)$}.

For $1\leq a\leq l$, we define $\ve_a=\hbar(k+M-N-u_a+q_a)$. In the case that $u_s-u_{s+1},q_s-q_{s+1}\geq 2$ and $u_s-q_s+q_s-q_{s+1}\geq5$, let us define the homomorphism
\begin{equation*}
\Delta^{s}\colon\ Y_{\hbar,\ve_s}\big(\widehat{\mathfrak{sl}}(u_s-u_{s+1}|q_s-q_{s+1})\big)\to\bigotimes_{1\leq a\leq s}Y_{\hbar,\ve_a}\big(\widehat{\mathfrak{sl}}(u_a|q_a)\big)
\end{equation*}
defined by
\begin{equation*}
\Delta^{s}=\Biggl(\prod_{a=1}^{s-1}  \big(\big(\big(\Psi_2^{q_{a+1}|u_{a+1},q_a|u_a}\otimes1\big)\circ\Delta\big)\otimes\id^{\otimes (s-a-1)}\big)\circ\Psi_1^{u_s-u_{s+1}|q_s-q_{s+1},u_s|q_s}\Biggr)\otimes\id^{\otimes (l-s)}.
\end{equation*}
\begin{Theorem}[{\cite[Theorem~11.1]{U13}}]\label{Main3}
There exists an algebra homomorphism
\begin{equation*}
\Phi_s\colon\ Y_{\hbar,\ve_s}\big(\widehat{\mathfrak{sl}}(u_s-u_{s+1}|q_s-q_{s+1})\big)\to \mathcal{U}\big(\mathcal{W}^{k}(\mathfrak{gl}(N),f)\big)
\end{equation*}
determined by
\begin{equation}
\bigotimes_{1\leq a\leq s}\ev_{\hbar,\ve_a}^{u_a|q_a,-x_a\hbar}\circ\Delta^{s}=\widetilde{\mu}\circ\Phi_s,\label{eqqq}
\end{equation}
where \smash{$x_a=\gamma_a+q_a-q_s-\frac{u_a-u_s}{2}$}.
\end{Theorem}
By \eqref{eqqq}, we find that
\begin{gather*}
\Phi_s\big(X^+_{i,0}\big)=\begin{cases}
W^{(1)}_{u_1-u_s+i,u_1-u_s+i+1}&\text{if }1\leq i\leq u_s-u_{s+1}-1,\\
W^{(1)}_{u_1-u_s+i,-q_1+q_s-1}&\text{if }i=u_s-u_{s+1},\\
W^{(1)}_{u_1-u_s+i,-q_1+q_s+i-1}&\text{if }-q_s+q_{s+1}+1\leq i\leq-1,\\
W^{(1)}_{-q_1+q_{s+1},u_1-u_s+1}t&\text{if }i=-q_s+q_{s+1},
\end{cases}\\
\Phi_s\big(X^-_{i,0}\big)=\begin{cases}
W^{(1)}_{u_1-u_s+i+1,u_1-u_s+i}&\text{if }1\leq i\leq u_s-u_{s+1}-1,\\
W^{(1)}_{-q_1+q_s-1,u_1-u_s+i}&\text{if }i=u_s-u_{s+1},\\
W^{(1)}_{-q_1+q_s+i-1,u_1-u_s+i}&\text{if }-q_s+q_{s+1}+1\leq i\leq-1,\\
W^{(1)}_{u_1-u_s+1,-q_1+q_{s+1}}t^{-1}&\text{if }i=-q_s+q_{s+1},
\end{cases}
\end{gather*}
and
\begin{gather*}
\Phi_l\big(\widetilde{H}_{i_+,1}\big)=\begin{cases}
-\hbar\big(W^{(2)}_{u_1-u_l+i_+,u_1-u_l+i_+}t-W^{(2)}_{u_1-u_l+i_++1,u_1-u_l+i_++1}t\big)\\
\displaystyle \quad{}-\frac{i_+}{2}\hbar\big(W^{(1)}_{u_1-u_l+i_+,u_1-u_l+i_+}-W^{(1)}_{u_1-u_l+i_++1,i_++1}\big)\\
\quad{}+U_{i_+}-U_{i_++1}&\text{if }1\leq i_+\leq u_l-1,\\
-\hbar\big(W^{(2)}_{u_1,u_1}+W^{(2)}_{-q_1+q_l-1,-q_1+q_l-1}\big)\\
\displaystyle \quad{}-\frac{u_l}{2}\hbar\big(W^{(1)}_{u_1,u_1}+W^{(1)}_{-q_1+q_l-1,-q_1+q_l-1}\big)\\
\quad{}+U_{u_l}-U_{-1}&\text{if } i_+=u_l,
\end{cases}\\
\Phi_l\big(\widetilde{H}_{i_-,1}\big)=\begin{cases}
\makebox[0pt][l]{$\hbar\big(W^{(2)}_{-q_1+q_l+i_-,-q_1+q_l+i_-}t-W^{(2)}_{-q_1+q_l+i_--1,-q_1+q_l+i_--1}t\big)$}&\\
\displaystyle \makebox[0pt][l]{$\displaystyle\quad{}+\frac{u_l+i_-}{2}\hbar\big(W^{(1)}_{-q_1+q_l+i_-,-q_1+q_l+i_-} -W^{(1)}_{-q_1+q_l+i_--1,-q_1+q_l+i_--1}\big)$}&\\
\quad{}+U_{i_-}-U_{i_--1}& \text{if }-q_l+1\leq i_-\leq -1,\\
\hbar\big(W^{(2)}_{-q_l,-q_l}t+W^{(2)}_{u_1-u_l+1,u_1-u_l+1}t\big)-\ve_l W^{(1)}_{-q_l,-q_l}&\\
\displaystyle \quad{}+\sum\limits_{u=1}^l \big(-\hbar\alpha_ux_u+\ve_u\alpha_u+\frac{\hbar}{2}(q_u-q_l)\alpha_u\big)&\\
\displaystyle \quad{}-\hbar\sum\limits_{u=1}^{u_1-u_l}  W^{(1)}_{u,u}-\hbar\sum\limits_{u=1}^{q_1-q_l}  W^{(1)}_{-u,-u}+U_{-q_l}-U_{1}& \text{if }i_-=-q_l,
\end{cases}
\end{gather*}
where
\begin{gather*}
U_{i_+}=-\frac{\hbar}{2}\big(W^{(1)}_{u_1-u_l+i_+,u_1-u_l+i_+}\big)^2+\hbar\sum\limits_{s\geq0} \sum\limits_{u=1}^{i_+}W^{(1)}_{u_1-u_l+i_+,u_1-u_l+u}t^{-s}W^{(1)}_{u_1-u_l+u,u_1-u_l+i_+}t^s\\
\hphantom{U_{i_+}=}{}+\hbar\sum\limits_{s\geq0} \sum\limits_{u=i_++1}^{u_l}W^{(1)}_{u_1-u_l+i_+,u_1-u_l+u}t^{-s-1}W^{(1)}_{u_1-u_l+u,u_1-u_l+i_+}t^{s+1}\\
\hphantom{U_{i_+}=}{}-\hbar\sum\limits_{s\geq0} \sum\limits_{u=-q_l}^{-1}W^{(1)}_{u_1-u_l+i_+,-q_1+q_l+u}t^{-s-1}W^{(1)}_{-q_1+q_l+u,u_1-u_l+i_+}t^{s+1},\\
U_{i_-}=-\frac{\hbar}{2}\big(W^{(1)}_{-q_1+q_l+i_-,-q_1+q_l+i_-}\big)^2-\hbar\sum\limits_{s\geq0} \sum\limits_{u=1}^{u_l}W^{(1)}_{-q_1+q_l+i_-,u_1-u_l+u}t^{-s}W^{(1)}_{u_1-u_l+u,-q_1+q_l+i_-}t^s\\
\hphantom{U_{i_-}=}{}-\hbar\sum\limits_{s\geq0} \sum\limits_{u=i_-}^{-1}W^{(1)}_{-q_1+q_l+i_-,-q_1+q_l+u}t^{-s}W^{(1)}_{-q_1+q_l+u,-q_1+q_l+i_-}t^{s}\\
\hphantom{U_{i_-}=}{}+\hbar\sum\limits_{s\geq0} \sum\limits_{u=-q_l}^{i_--1}W^{(1)}_{-q_1+q_l+i_-,-q_1+q_l+u}t^{-s-1}W^{(1)}_{-q_1+q_l+u,-q_1+q_l+i_-}t^{s+1}.
\end{gather*}
In the rectangular case, $\Phi_l$ coincides with the homomorphism $\Phi$ in \cite[Theorem~5.1]{U4}. Hereafter, we denote the homomorphism $\Phi_l$ in the rectangular case by $\Phi^{u_1|q_1}$.
\begin{Theorem}\qquad\label{thm1}
\begin{enumerate}
\item[$(1)$] Assume that $m_1,m_2,n_1,n_2\geq0$, $m_1+n_1,m_2+n_2\geq 5$ and
\begin{equation*}
\ve_l=\frac{k+(m_1+m_2)-(n_1+n_2)}{\hbar}.
\end{equation*}
Then, we obtain the relation:
\begin{equation*}
\iota_2\circ\Phi^{m_1|n_1}=\Phi^{m_1+m_2|n_1+n_2}\circ\Psi_1^{m_1|n_1,m_1+m_2|n_1+n_2}.
\end{equation*}

\item[$(2)$] We also suppose the condition that $\ve_l\neq0$. Then, we have a homomorphism
\begin{align*}
&\Phi^{m_1+m_2|n_1+n_2}\circ\Psi_2^{m_2|n_2,m_1+m_2|n_1+n_2}\colon\\
& \qquad Y_{\hbar,\ve+(m_1-n_1)\hbar}\big(\widehat{\mathfrak{sl}}(m_2|n_2)\big)
\to C(\mathcal{U}(W_1),\mathcal{U}(W_2)),
\end{align*}
where
\begin{align*}
&W_1=W^k\big(\mathfrak{gl}(2(m_1+m_2),2(n_1+n_2))|\big(2^{m_1+m_2|n_1+n_2}\big)\big),\\
&W_2=W^{k+m_2-n_2}\big(\mathfrak{gl}(2(m_1+m_2)|2(n_1+n_2)),\big(2^{m_1|n_1}\big)\big).
\end{align*}
\end{enumerate}
\end{Theorem}
\begin{proof}
(1) follows from the definition of $\Phi^{m|n}$ and $\Psi_1^{m_1|n_1,m_1+m_2|n_1+n_2}$. By (1), Corollary~\ref{Cor}\,(1) and~\cite[Theorem~5.1]{U4}, the completion of the image of $\iota_2\circ\Phi^{m_1|n_1}$ coincides with the universal enveloping algebra \smash{$\mathcal{U}\big(W^{k+m_2-n_2}\big(\mathfrak{gl}(2(m_1+m_2)|2(n_1+n_2)),\big(2^{m_1|n_1}\big)\big)\big)$}. Then, by Theorem~\ref{Commutativity}, we obtain (2).
\end{proof}

\begin{Theorem}\label{cent}\qquad
\begin{enumerate}
\item[$(1)$] Assume that $u_1-u_2,q_1-q_2>0$, $u_1-u_2+q_1-q_2\geq 5$ and $\ve_1\neq 0$. Then, we find that
\begin{gather*}
\Phi_s\colon\ Y_{\hbar,\ve_s}\big(\widehat{\mathfrak{sl}}(u_s-u_{s+1}|q_s-q_{s+1})\big)\\
\hphantom{\Phi_s\colon\ }{}\quad\to C\big(\mathcal{U}\big(\mathcal{W}^k(\mathfrak{gl}(M|N),f)\big),\mathcal{U}\big(\widehat{\mathfrak{gl}}(u_1-u_2|q_1-q_2)\big)\big)
\end{gather*}
for $s\neq1$.

\item[$(2)$] Assume that $u_1=u_2>u_3,q_1=q_2>q_3$, $u_1-u_3+q_1-q_3\geq 5$ and $\ve_2\neq 0$. Then, for $s\neq1$, we obtain
\begin{equation*}
\Phi_s\colon\ Y_{\hbar,\ve_s}\big(\widehat{\mathfrak{sl}}(u_s-u_{s+1}|q_s-q_{s+1})\big)\to C(\mathcal{U}(W_3),\mathcal{U}(W_4)),
\end{equation*}
where
\begin{align*}
&\widetilde{k}=k+M-N-2(u_1-q_1)+(u_3-q_3),\\
&W_3=\mathcal{W}^k(\mathfrak{gl}(M|N),f)),\\
&W_4=\mathcal{W}^{\widetilde{k}}\big(\mathfrak{gl}(2(u_1-u_3)|2(q_1-q_3)),\big(2^{u_1-u_3|q_1-q_3}\big)\big).
\end{align*}
where $\widehat{k}=k+M-N-2(u_1-q_1)+u_3-q_3$.
\end{enumerate}
\end{Theorem}
\begin{proof}
(1) By Theorem~\ref{Main3}, the image of $\Phi_s$ is contained in the completion of the tensor of \smash{$\ev_{\hbar,\ve_1}^{u_1|q_1,-x_1\hbar}\circ\Psi_2^{u_2|v_2,u_1|v_1}$} and \smash{$\widehat{\bigotimes}_{2\leq a\leq l}U\big(\widehat{\mathfrak{gl}}(u_a|q_a)\big)$}. By Theorem~\ref{thm:main}\,(2) and Corollary~\ref{Cor}\,(2), the~completion of the image of $\Psi_1$ coincides with \smash{$\mathcal{U}(\widehat{\mathfrak{gl}}(u_1-u_2|q_1-q_2))$} and is contained in \smash{$\ev_{\hbar,\ve_1}^{u_1|q_1,-x_1\hbar}\circ\Psi_1^{u_1-u_2|v_1-v_2,u_1|v_1}$}. Thus, the image of $\Psi_s$ is contained in the centralizer algebra.

(2) Similarly to (1), we can prove by using \cite[Theorem~5.1]{U4} and
Corollary~\ref{Cor}\,(3) instead of Theorem~\ref{thm:main}\,(2) and Corollary~\ref{Cor}\,(2).
\end{proof}

\begin{Remark}
In Section~\ref{section2}, we give a definition of the affine super Yangian in the case $m,n\geq2$ and $m+n\geq5$ in order to use the finite presentation given in Theorem~\ref{Prop32}. This is why we assume the condition $m_1+n_1,m_2+n_2\geq5$ in Theorem~\ref{thm1} and $u_1-u_2+q_1-q_2\geq5$ or $u_1-u_3+q_1-q_3\geq5$ in Theorem~\ref{cent}.
\end{Remark}
For a vertex algebra $A$ and its subalgebra $B$, we set the coset vertex algebra $C(A,B)$ as
\begin{align*}
C(A,B)=\big\{x\in A\mid b_{(s)}x=0\text{ for }b\in B,\,s\geq0\big\}.
\end{align*}
Similarly to \cite[Theorem~6.5]{U10} and \cite[Theorem~7,7]{U12}, Theorem~\ref{cent} induces the following.
\begin{Theorem}\qquad
\begin{enumerate}
\item[$(1)$] Assume that $m_1,m_2,n_1,n_2\geq0$, $m_1+n_1,m_2+n_2\geq 5$ and
\begin{equation*}
\ve_l=\frac{k+(m_1+m_2)-(n_1+n_2)}{\hbar}\neq0.
\end{equation*}
Then, we obtain a homomorphism
\begin{align*}
\Phi^{m_1+m_2|n_1+n_2}&\circ\Psi_2^{m_2|n_2,m_1+m_2|n_1+n_2}\colon\ Y_{\hbar,\ve_l+(m_1-n_1)\hbar}\big(\widehat{\mathfrak{sl}}(m_2|n_2)\big)\to\mathcal{U}(C(W_1,w_2)),
\end{align*}
where $w_2=\mathcal{W}^{k+m_2-n_2}\big(\mathfrak{sl}(2(m_1+m_2)|2(n_1+n_2)),\big(2^{m_1+m_2|n_1+n_2}\big)\big)$.
\item[$(2)$] Assume that $u_1-u_2,q_1-q_2>0$,
$u_1-u_2+q_1-q_2\geq 5$ and $\ve_1\neq 0$. Then, we have a~homomorphism
\begin{gather*}
\Phi_s\colon\ Y_{\hbar,\ve_s}\big(\widehat{\mathfrak{sl}}(u_s-u_{s+1}|q_s-q_{s+1})\big)\\
\hphantom{\Phi_s\colon\ }{}\quad
\to\mathcal{U}\big(C\big(\mathcal{W}^k(\mathfrak{gl}(M|N),f\big),V^\kappa_1(\mathfrak{sl}(u_1-u_2|q_1-q_2)))\big)
\end{gather*}
for $s\geq1$.
\item[$(3)$] Assume that $u_1=u_2>u_3$, $q_1=q_2>q_3$, $u_1-u_3+q_1-q_3\geq 5$ and $\ve_2\neq 0$. Then, the homomorphism $\Phi_s$ induces
\begin{equation*}
\Phi_s\colon\ Y_{\hbar,\ve_s}\big(\widehat{\mathfrak{sl}}(u_s-u_{s+1}|q_s-q_{s+1})\big)\to\mathcal{U}(C(W_3,w_4)),
\end{equation*}
where $w_4=\mathcal{W}^{\widetilde{k}}\big(\mathfrak{sl}(2(u_1-u_3)|2(q_1-q_3)),\big(2^{u_1-u_3|q_1-q_3}\big)\big)$.
\end{enumerate}
\end{Theorem}

\section{Extended affine super Yangian}\label{section7}
We extend the definition of the affine super Yangian to the new associative algebra. Let us set
\smash{$\widehat{\mathfrak{sl}}(m_2|n_2)^R$} as a Lie subalgebra of \smash{$\widehat{\mathfrak{sl}}(m_1+m_2|n_1+n_2)=\mathfrak{sl}(m_1+m_2|n_1+n_2)\otimes\mathbb{C}\big[t^{\pm1}\big]\oplus\mathbb{C}c$} generated by \smash{$\big\{E_{i,j}t^s\mid s\in\mathbb{Z},i\in I_{m_1+m_2|n_1+n_2},j\in I_{m_1+m_2|n_1+n_2}\setminus I_{m_1|n_1}\big\}$} and $c$.
\begin{Definition}
Let $m_1,n_1\geq 0$. We define \smash{$Y^{m_1+m_2|n_1+n_2}_{\hbar,\ve}\big(\widehat{\mathfrak{sl}}(m_2|n_2)\big)$} by the associative algebra whose generators are
\[
\big\{H_{i,r},X^\pm_{i,r}\mid0\leq i\leq n-1,\,r\in\mathbb{Z}_{\geq0}\big\}
\]
 and \smash{$\widehat{\mathfrak{sl}}(m_2|n_2)^R$} with the relations \eqref{Eq2.1}--\eqref{Eq2.10} and we identify with $H_{i,0}$ and \smash{$X^\pm_{i,0}$} with
 \[
 \Psi^{m_2|n_2,m_1+m_2|n_1+n_2}_2(H_{i,0}) \qquad \text{and} \qquad \Psi^{m_2|n_2,m_1+m_2|n_1+n_2}_2\big(X^\pm_{i,0}\big)
 \qquad \text{for \smash{$i\in I_{m_2|n_2}$}}.
 \]
\end{Definition}

We set the degree on \smash{$Y^{m_1+m_2|n_1+n_2}_{\hbar,\ve}\big(\widehat{\mathfrak{sl}}(m_2|n_2)\big)$} as
\begin{alignat*}{3}
&{\rm deg}(H_{i,r})=0,&&\qquad {\rm deg}\big(X^\pm_{i,r}\big)=\pm \delta_{i,0},&\\
&{\rm deg}(xt^s)=s,&&\qquad {\rm deg}(c_{m+n})=0\qquad \text{for }xt^s\in\widehat{\mathfrak{sl}}(m_2|n_2)^R.&
\end{alignat*}
Using this degree, we denote the standard degreewise completion of \smash{$Y^{m_1+m_2|n_1+n_2}_{\hbar,\ve}\big(\widehat{\mathfrak{sl}}(m_2|n_2)\big)$} by \smash{$\widetilde{Y}^{m_1+m_2|n_1+n_2}_{\hbar,\ve}\big(\widehat{\mathfrak{sl}}(m_2|n_2)\big)$}.

For $1\leq v_+\leq m_1$, $-n_1\leq v_-\leq-1$, $1\leq i_+,j_+\leq m_2$ and $-n_2\leq i_-,j_-\leq -1$, let us set
\begin{gather*}
a^{v_+,w}_{i_+,j_+}=\delta(j_+<i_+)\hbar\sum\limits_{s\geq0}  E_{v_+,i_++m_1}t^{w-s-1}E_{i_++m_1,j_++m_1}t^{s+1}\\
\hphantom{a^{v_+,w}_{i_+,j_+}=}{}+\delta(j_+>i_+)\hbar\sum\limits_{s\geq0}  E_{v_+,i_++m_1}t^{w-s}E_{i_++m_1,j_++m_1}t^{s},\\
a^{v_-,w}_{i_+,j_+}=\delta(j_+<i_+)\hbar\sum\limits_{s\geq0}  E_{v_-,i_++m_1}t^{w-1-s}E_{i_++m_1,j_++m_1}t^{s+1}\\
\hphantom{a^{v_-,w}_{i_+,j_+}=}{}+\delta(j_+>i_+)\hbar\sum\limits_{s\geq0}  E_{v_-,i_++m_1}t^{w-s}E_{i_++m_1,j_++m_1}t^{s},\\
a^{v_+,w}_{i_+,j_-}=\hbar\sum\limits_{s\geq0}  E_{v_+,i_++m_1}t^{w-s}E_{i_++m_1,-j_--n_1}t^{s},\\
a^{v_-,w}_{i_+,j_-}=\hbar\sum\limits_{s\geq0}  E_{v_-,i_++m_1}t^{w-s}E_{i_++m_1,-j_--n_1}t^{s},\\
a^{v_+,w}_{i_-,j_+}=\hbar\sum\limits_{s\geq0}  E_{v_+,i_--n_1}t^{w-s-1}E_{i_--n_1,j_++m_1}t^{s+1},\\
a^{v_-,w}_{i_-,j_+}=\hbar\sum\limits_{s\geq0}  E_{v_-,i_--n_1}t^{w-s-1}E_{i_--n_1,j_++m_1}t^{s+1},\\
a^{v_+,w}_{i_-,j_-}=\delta(j_->i_-)\hbar\sum\limits_{s\geq0}  E_{v_+,i_--n_1}t^{w-s-1}E_{i_--n_1,-j_--n_1}t^{s+1}\\
\hphantom{a^{v_+,w}_{i_-,j_-}=}{}+\delta(j_-<i_-)\hbar\sum\limits_{s\geq0}  E_{v_+,i_--n_1}t^{w-s}E_{i_--n_1,-j_--n_1}t^{s},\\
a^{v_-,w}_{i_-,j_-}=\delta(j_->i_-)\hbar\sum\limits_{s\geq0}  E_{v_-,i_--n_1}t^{w-s-1}E_{i_--n_1,-j_--n_1}t^{s+1}\\
\hphantom{a^{v_-,w}_{i_-,j_-}=}{}+\delta(j_-<i_-)\hbar\sum\limits_{s\geq0}  E_{v_-,i_--n_1}t^{w-s}E_{i_--n_1,-j_--n_1}t^{s}.
\end{gather*}
We set \smash{$Y^{m_1+m_2|n_1+n_2,R}_{\hbar,\ve}\big(\widehat{\mathfrak{sl}}(m_2|n_2)\big)$} as a quotient algebra of \smash{$\widetilde{Y}^{m_1+m_2|n_1+n_2}_{\hbar,\ve}\big(\widehat{\mathfrak{sl}}(m_2|n_2)\big)$} divided by
\begin{alignat}{3}
&\bigl[\widetilde{H}_{i_+,1},E_{v,m_1+j_+}t^w\bigr]=a^{v,w}_{i_+,j_+}-a^{v,w}_{i_++1,j_+}&&\qquad \text{for }j_+\neq i_+,i_++1,&\label{Eq2.11-1}\\
&\bigl[\widetilde{H}_{i_+,1},E_{v_+,-n_1+j_-}t^w\bigr]=a^{v,w}_{i_+,j_-}-a^{v,w}_{i_++1,j_-},&\\
&\bigl[\widetilde{H}_{i_-,1},E_{v,-n_1+j_-}t^w\bigr]=a^{v,w}_{i_-,j_-}-a^{v,w}_{i_--1,j_-}&&\qquad\text{for }j_-\neq i_-,i_--1,&\\
&\bigl[\widetilde{H}_{i_-,1},E_{v,m_1+j_+}t^w\bigr]=a^{v,w}_{i_-,j_+}-a^{v,w}_{i_--1,j_+},&\\
&\bigl[\widetilde{H}_{i_+-1,1},E_{v,m_1+i_+}t^w\bigr]+\bigl[\widetilde{H}_{i_+,1},E_{v,m_1+i_+}t^w\bigr]&\nonumber\\
&\qquad=a^{v,w}_{i_+-1,i_+}-a^{v,w}_{i_++1,i_+}-\frac{\hbar}{2}E_{v_+,m_1+i_+}t^w && \qquad\text{for }i_+\neq 1,m_2,&\\
&\bigl[\widetilde{H}_{i_-+1,1},E_{v,-n_1+i_-}t^w\bigr]+\bigl[\widetilde{H}_{i_-,1},E_{v,-n_1+i_-}t^w\bigr]&\nonumber\\
&\qquad=a^{v,w}_{i_-+1,i_-}-a^{v_+,w}_{i_--1,i_-}-\frac{\hbar}{2}E_{v,-n_1+i_-}t^w&& \qquad \text{for }i_-\neq-1,-n_2,&\\
&\bigl[\widetilde{H}_{m_2,1},E_{v,m_1+j_+}t^w\bigr]=a^{v,w}_{m_2,j_+}-a^{v,w}_{-1,j_+}-\hbar E_{v,m_1+j_+}t^w &&\qquad \text{for }j_+\neq m_2,&\\
&\bigl[\widetilde{H}_{m_2,1},E_{v,-n_1+j_-}t^w\bigr]=a^{v,w}_{m_2,j_-}-a^{v,w}_{-1,j_-}&&\qquad \text{for }j_-\neq -1,&\\
&\bigl[\widetilde{H}_{m_2-1,1},E_{v,m_1+m_2}t^w\bigr]+\bigl[\widetilde{H}_{m_2,1},E_{v,m_1+m_2}t^w\bigr]&\\
&\qquad=a^{v,w}_{m_2-1,m_2}-a^{v,w}_{-1,m_2}-\frac{\hbar}{2}E_{v,m_1+m_2}t^w,&\\
&\makebox[0pt][l]{$\bigl[\widetilde{H}_{-1,1},E_{v,-1}t^w\bigr]+\bigl[\widetilde{H}_{m_2,1},E_{v,-1}t^w\bigr]=a^{v,w}_{m_2,-1}-a^{v,w}_{-2,-1}-\frac{\hbar}{2}E_{v,-1}t^w,$}&\\
&\bigl[\widetilde{H}_{0,1},E_{v_+,m_1+j_+}t^w\bigr]=a^{v,w}_{-n_2,j_+}-a^{v,w}_{1,j_+}&& \qquad \text{for }j_+\neq1,&\\
&\bigl[\widetilde{H}_{0,1},E_{v,-n_1+j_-}t^w\bigr]=a^{v,w}_{-n_2,j_-}-a^{v,w}_{1,j_-}&& \qquad \text{for }j_-\neq-n_2,&\\
&\makebox[0pt][l]{$\bigl[\widetilde{H}_{-n_2+1,1},E_{v,-n_1-n_2}t^w\bigr]+\bigl[\widetilde{H}_{0,1},E_{v,-n_1-n_2}t^w\bigr]$}&\nonumber\\
&\makebox[0pt][l]{$\qquad=a^{v,w}_{-n_2+1,-n_2}-a^{v,w}_{1,-n_2}+\frac{\hbar}{2}(m_2-n_2+1)E_{v,m_1+1}t^w+\ve E_{v,m_1+1}t^w,$}&\\
&\makebox[0pt][l]{$\bigl[\widetilde{H}_{1,1},E_{v,m_1+1}t^w\bigr]+\bigl[\widetilde{H}_{0,1},E_{v,m_1+1}t^w\bigr]=a^{v,w}_{-n_2,1}-a^{v,w}_{2,1}-\frac{\hbar}{2}E_{v,m_1+1}t^w$}& \label{Eq2.19-1}
\end{alignat}
for $1\leq i_+,j_+\leq m_2$, $-n_2\leq i_-,j_-\leq-1$ and $w\in\mathbb{Z}$.
\begin{Theorem}
There exists a homomorphism
\begin{gather*}
\Psi^{m_2|n_2,m_1+m_2|n_1+n_2,R}_2\colon\\
\qquad Y_{\hbar,\ve}^{m_1+m_2|n_1+n_2,R}\bigl(\widehat{\mathfrak{sl}}(m_2|n_2)\bigr)
\to\widetilde{Y}_{\hbar,\ve-(m_1-n_1)\hbar}\bigl(\widehat{\mathfrak{sl}}(m_1+m_2|n_1+n_2)\bigr)
\end{gather*}
given by
\begin{gather*}
\Psi^{m_2|n_2,m_1+m_2|n_1+n_2,R}_2(y)=y \qquad \text{for } y\in\widehat{\mathfrak{sl}}(m_1+m_2|n_1+n_2),\\
\Psi^{m_2|n_2,m_1+m_2|n_1+n_2,R}_2(Z_{i,r})=\Psi^{m_2|n_2,m_1+m_2|n_1+n_2}_2(Z_{i,r}) \qquad \text{for } Z=H,X^\pm \text{ and } r=0,1.
\end{gather*}
\end{Theorem}
\begin{proof}
It is enough to show the compatibility with \eqref{Eq2.11-1}--\eqref{Eq2.19-1}. By \eqref{J}, we find that the relations replacing \smash{$\widetilde{H}_{i,1}$} with \smash{$\Psi^{m_1+m_2|n_1+n_2}_2\big(\widetilde{H}_{i,1}\big)$} in \eqref{Eq2.11-1}--\eqref{Eq2.19-1} yields the same result as replacing~\smash{$\widetilde{H}_{i,1}$} with
\[
\ev_{\hbar,\ve-(m_1-n_1)\hbar}^{m_1+m_2|n_1+n_2,0}\big(\Psi^{m_1+m_2|n_1+n_2}_2\big(\widetilde{H}_{i,1}\big)\big).
\] Thus, it follows from a direct computation.
\end{proof}

\begin{Theorem}
There exists a homomorphism
\begin{gather*}
\Delta^{m_2|n_2}\colon\ Y_{\hbar,\ve}\big(\widehat{\mathfrak{sl}}(m_2|n_2)\big)\\
\hphantom{\Delta^{m_2|n_2}\colon\ }{}\quad \to Y_{\hbar,\ve-(m_1-n_1)\hbar}\big(\widehat{\mathfrak{sl}}(m_1+m_2|n_1+n_2)\big)\widehat{\otimes}Y^{m_1+m_2|n_1+n_2,R}_{\hbar,\ve}\big(\widehat{\mathfrak{sl}}(m_2|n_2)\big)
\end{gather*}
determined by
\begin{align*}
&\Delta^{m_2|n_2}(y)=1\otimes y+y\otimes 1\text{for }y\in\widehat{\mathfrak{sl}}(m_1+m_2|n_1+n_2),\\
&\Delta^{m_2|n_2}\big(X^+_{i,1}\big)=\begin{cases}
(\Psi^{m_2|n_2,m_1+m_2|n_1+n_2,R}\otimes\id)\circ\Delta\big(X^+_{i,1}\big)+Y^0_{i,i+1} \\
\hspace*{50mm}\text{for }1\leq i_+\leq m_2-1,\\
(\Psi^{m_2|n_2,m_1+m_2|n_1+n_2,R}\otimes\id)\circ\Delta\big(X^+_{m_2,1}\big)+Y^0_{m_2,-1}\\
\hspace*{50mm}\text{for }i=m_2,\\
(\Psi^{m_2|n_2,m_1+m_2|n_1+n_2,R}\otimes\id)\circ\Delta\big(X^+_{i,1}\big)+Y^0_{i,i-1} \\
\hspace*{50mm}\text{for }-n_2+1\leq i\leq -1,\\
(\Psi^{m_2|n_2,m_1+m_2|n_1+n_2,R}\otimes\id)\circ\Delta\big(X^+_{-n_2,1}\big)+Y^1_{-n_2,1}\\
\hspace*{50mm}\text{for }i=-n_2,
\end{cases}\\
&\Delta^{m_2|n_2}\big(X^-_{i,1}\big)=\begin{cases}
(\Psi^{m_2|n_2,m_1+m_2|n_1+n_2,R}\otimes\id)\circ\Delta\big(X^-_{i,1}\big)+Y^0_{i+1,i}\\
\hspace*{50mm}\text{for }1\leq i_+\leq m_2-1,\\
(\Psi^{m_2|n_2,m_1+m_2|n_1+n_2,R}\otimes\id)\circ\Delta\big(X^-_{m_2,1}\big)+Y^0_{-1,m_2}\\
\hspace*{50mm}\text{for }i=m_2,\\
(\Psi^{m_2|n_2,m_1+m_2|n_1+n_2,R}\otimes\id)\circ\Delta\big(X^-_{i,1}\big)-Y^0_{i-1,i}\\
\hspace*{50mm}\text{for }-n_2+1\leq i\leq -1,\\
(\Psi^{m_2|n_2,m_1+m_2|n_1+n_2,R}\otimes\id)\circ\Delta\big(X^-_{-n_2,1}\big)-Y^{-1}_{1,-n_2}\\
\hspace*{50mm}\text{for }i=-n_2,
\end{cases}\\
&\Delta^{m_2|n_2}\big(X^+_{i,1}\big)=\begin{cases}
(\Psi^{m_2|n_2,m_1+m_2|n_1+n_2,R}\otimes\id)\circ\Delta\big(H_{i,1}\big)+Y^0_{i,i}-Y^0_{i+1,i+1}\\
\hspace*{50mm}\text{for }1\leq i_+\leq m_2-1,\\
(\Psi^{m_2|n_2,m_1+m_2|n_1+n_2,R}\otimes\id)\circ\Delta\big(H_{m_2,1}\big)+Y^0_{m_2,m_2}+Y^0_{1,1}\\
\hspace*{50mm}\text{for }i=m_2,\\
(\Psi^{m_2|n_2,m_1+m_2|n_1+n_2,R}\otimes\id)\circ\Delta\big(H_{i,1}\big)-Y^0_{i,i}+Y^0_{i-1,i-1}\\
\hspace*{50mm}\text{for }-n_2+1\leq i\leq -1,\\
(\Psi^{m_2|n_2,m_1+m_2|n_1+n_2,R}\otimes\id)\circ\Delta\big(H_{-n_2,1}\big)-Y^0_{-n_2,-n_2}-Y^0_{1,1}\\
\hspace*{50mm}\text{for }i=-n_2,
\end{cases}
\end{align*}
where we set
\begin{align*}
Y_{i,j}^r&=\hbar\sum\limits_{\substack{s\in\mathbb{Z}\\u\in I_{m_1|n_1}}}  (-1)^{p(u)}E_{i,u}t^{-s}\otimes E_{u,j}t^{s+r}.
\end{align*}
\end{Theorem}
\begin{proof}
The compatibilities with \eqref{Eq2.2}--\eqref{Eq2.12} follows from a direct computation. It is enough to show the compatibility with \eqref{Eq2.1}.
Since we obtain
\begin{align*}
\big(\id\otimes\Psi^{m_2|n_2,m_1+m_2|n_1+n_2,R}\big)\circ\Delta^{m_2|n_2}=\Delta\circ\Psi^{m_2|n_2,m_1+m_2|n_1+n_2}_2
\end{align*}
by a direct computation, we have
\begin{align*}
&\id\otimes\Psi^{m_2|n_2,m_1+m_2|n_1+n_2,R}\big(\big[\Delta^{m_2|n_2}(H_{i,1}),\Delta^{m_2|n_2}(H_{j,1})\big]\big)\nonumber\\
&\qquad=\Delta\big(\big[\Psi^{m_2|n_2,m_1+m_2|n_1+n_2}_2(H_{i,1}),\Psi^{m_2|n_2,m_1+m_2|n_1+n_2}(H_{j,1})\big]\big)=0.
\end{align*}
Using \eqref{Eq2.11-1}--\eqref{Eq2.19-1}, we can write down $\big[\Delta^{m_2|n_2}(H_{i,1}),\Delta^{m_2|n_2}(H_{j,1})\big]$ as an element of the completion of \smash{$\bigotimes^2U\big(\widehat{\mathfrak{sl}}(m_1+m_2|n_1+n_2)\big)$} since \smash{$U\big(\widehat{\mathfrak{sl}}(m_1+m_2|n_1+n_2)\big)$} can be embedded into \smash{$\bigotimes^2 Y_{\hbar,\ve-(m_1-n_1)\hbar}\big(\widehat{\mathfrak{sl}}(m_1+m_2|n_1+n_2)\big)$}.
\end{proof}
\begin{Theorem}\label{Main4}
Assume that $\ve=k+M-N-u_l+q_l$. There exists a homomorphism
\begin{align*}
\Phi^R\colon\ Y^{u_1|q_1,R}_{\hbar,\ve}\big(\widehat{\mathfrak{sl}}(u_l|q_l)\big)\to\mathcal{U}\big(\mathcal{W}^k(\mathfrak{gl}(M|N),f)\big)
\end{align*}
determined by
\begin{alignat*}{3}
&\Phi^R(E_{v,i}t^s)=W^{(1)}_{v,i}t^s&& \qquad \text{for } v\in I_{u_1|u_l|q_1-q_l},\ i\in I_{u_1|q_1}\setminus I_{u_1-u_l|q_1-q_l} \qquad\text{and}&\\
&\Phi^R(Z_{i,r})=\Phi(Z_{i,r})&&\qquad \text{for }i\in I_{u_l|q_l},\ r=0,1.&
\end{alignat*}
\end{Theorem}
\begin{proof}
Compatibility with \eqref{Eq2.1}--\eqref{Eq2.10} follows from Theorem~\ref{Main3}.
By a direct computation, we obtain
\begin{gather*}
\big(W^{(1)}_{v.j}\big)_{(0)}W^{(2)}_{i,i}=(-1)^{p(i)}\big(W^{(1)}_{v,i}\big)_{(-1)}W^{(1)}_{i,j},\\
\big(W^{(1)}_{v.j}\big)_{(1)}W^{(2)}_{i,i}=-(-1)^{p(i)}W^{(1)}_{p,j},\\
\big(W^{(1)}_{v.j}\big)_{(r)}W^{(2)}_{i,i}=0 \qquad \text{if } r\geq2.
\end{gather*}
for $v\in I_{u_1-u_l|q_1-q_l}$, $i,j\not\in I_{u_1-u_l|q_1-q_l}$, $i\neq j$.
By a direct computation, we obtain
\begin{gather*}
\hbar\big(W^{(1)}_{v,u_1-u_l+i_+}\big)_{(-1)}W^{(1)}_{u_1-u_l+i_+,u_1-u_l+j_+}t^{w+1}+\big[U_{i_+},W^{(1)}_{v,u_1-u_l+j_+}t^w\big]\\
\qquad=\hbar\delta(j_+<i_+)\sum\limits_{s\geq0} W^{(1)}_{v,u_1-u_l+i_+}t^{w-s-1}W^{(1)}_{u_1-u_l+i_+,u_1-u_l+j_+}t^{s+1}\\
\qquad\quad{}+\hbar\delta(j_+>i_+)\sum\limits_{s\geq0} W^{(1)}_{v,u_1-u_l+i_+}t^{w-s-1}W^{(1)}_{u_1-u_l+i_+,u_1-u_l+j_+}t^{s+1}\\
\qquad\quad{}-\hbar\delta(j_+<i_+)wW^{(1)}_{v,u_1-u_l+j_+}t^w-\hbar\delta(j_+>i_+)(w+1)W^{(1)}_{v,u_1-u_l+j_+}t^w,\\
\hbar\big(W^{(1)}_{v,u_1-u_l+i_+}\big)_{(-1)}W^{(1)}_{u_1-u_l+i_+,-q_1+q_l+j_-}t^{w+1}+\big[U_{i_+},W^{(1)}_{v,-q_1+q_l+j_-}t^w\big]\\
\qquad=\hbar\sum\limits_{s\geq0}  W^{(1)}_{v,u_1-u_l+i_+}t^{w-s}W^{(1)}_{u_1-u_l+i_+,-q_1+q_l+j_-}t^{s}-(w+1)\hbar W^{(1)}_{v,-q_1+q_l+j_-}t^{w},\\
\hbar\big(W^{(1)}_{v,-q_1+q_l+i_-}\big)_{(-1)}W^{(1)}_{-q_1+q_l+i_-,u_1-u_l+j_+}t^{w+1}+[U_{i_-},W^{(1)}_{v,u_1-u_l+j_+}t^w]\\
\qquad=\hbar\sum\limits_{s\geq0}   W^{(1)}_{v,-q_1+q_l+i_-}t^{w-s-1}W^{(1)}_{-q_1+q_l+i_-,u_1-u_l+j_+}t^{s+1}-w\hbar W^{(1)}_{v,u_1-u_l+j_+}t^{w}\\
\qquad\quad{}-\hbar(W^{(1)}_{v,-q_1+q_l+i_-})_{(-1)}W^{(1)}_{-q_1+q_l+i_-,-q_1+q_l+j_-}t^{w+1}+\big[U_{i_-},W^{(1)}_{v,-q_1+q_l+j_-}t^w\big]\\
\qquad=\delta(i_-<j_-)\hbar\sum\limits_{s\geq0}  W^{(1)}_{v,-q_1+q_l+i_-}t^{w-s-1}W^{(1)}_{-q_1+q_l+i_-,-q_1+q_l+j_-}t^{s+1}\\
\qquad\quad{}+\delta(i_->j_-)\hbar\sum\limits_{s\geq0} W^{(1)}_{v,-q_1+q_l+i_-}t^{w-s}W^{(1)}_{-q_1+q_l+i_-,-q_1+q_l+j_-}t^{s}\\
\qquad\quad{}-w\delta(i_-<j_-)\hbar W^{(1)}_{v,-q_1+q_l+j_-}t^w-(w+1)\delta(i_->j_-)\hbar W^{(1)}_{v,-q_1+q_l+j_-}t^w.
\end{gather*}
By using these relations, we can prove the compatibility with \eqref{Eq2.11-1}--\eqref{Eq2.19-1}.
\end{proof}
\section[Compatibility of Phi\_l with the parabolic induction for a W-superalgebra in the special setting]{Compatibility of $\boldsymbol{\Phi_l}$ with the parabolic induction\\ for a $\boldsymbol{W}$-superalgebra in the special setting}\label{section8}
In this section, we assume that $u_1>u_2>\dots>u_l>0$, $q_1>q_2>\dots>q_l>0$, $u_l,q_l\geq 2$ and~${u_l\neq q_l}$. In this case, we can give generators of the $W$-algebra $\mathcal{W}^k(\mathfrak{gl}(M|N),f)$.
\begin{Theorem}\label{gens2}
For \smash{$a\in I_{u_1-u_s|q_1-q_s}\setminus I_{u_1-u_{s-1}|q_1-q_{s-1}}$} and \smash{$b\in I_{u_1-u_v|q_1-q_v}\setminus I_{u_1-u_{v-1}|q_1-q_{v-1}}$},
the following elements are contained in $\mathcal{W}^k(\mathfrak{gl}(M|N),f)$:
\begin{gather*}
W^{(1)}_{a,b}=\sum\limits_{1\leq r\leq s}e^{(r)}_{a,b}\text{for }s\leq v,\\
W^{(2)}_{a,b}=
\sum\limits_{\substack{\col(i)=\col(j)+1\\\row(i)=a,\,\row(j)=b}}e_{i,j}-\sum\limits_{1\leq r\leq s}\Gamma_{r}e^{(r)}_{a,b}[-2]\\
\hphantom{W^{(2)}_{a,b}=}{}
+\sum\limits_{\substack{1\leq r_1<r_2\leq s\\x> u_1-u_s}} (-1)^{p(x)+p(e_{i,v})p(e_{x,j})}e^{(r_1)}_{x,b}[-1]e^{(r_2)}_{a,x}[-1]\\
\hphantom{W^{(2)}_{a,b}=}{}
+\sum\limits_{\substack{1\leq r_1<r_2\leq s\\x< -q_1+q_s}} (-1)^{p(x)+p(e_{a,x})p(e_{b,x})}e^{(r_1)}_{x,b}[-1]e^{(r_2)}_{a,x}[-1]\\
\hphantom{W^{(2)}_{a,b}=}{}
-\sum\limits_{\substack{r_1\geq r_2\\q_s-q_1\leq x\leq q_{r_1}-q_1\\\row(i)=a,\, \row(j)=b}} (-1)^{p(x)+p(e_{a,x})p(e_{x,b})}e^{(r_1)}_{x,b}[-1]e^{(r_2)}_{a,x}[-1]\\
\hphantom{W^{(2)}_{a,b}=}{}
-\sum\limits_{\substack{r_1\geq r_2\\u_1-u_{r_1}\leq x\leq u_1-u_s}}  (-1)^{p(x)+p(e_{a,x})p(e_{x,b})}e^{(r_1)}_{x,b}[-1]e^{(r_2)}_{a,x}[-1]\qquad \text{if }s=v\pm1,
\end{gather*}
where we set
\[
\Gamma_r=\begin{cases}
\gamma_r&\text{if }s=v,v-1,\\
\delta(r\leq s-1)\gamma_r+\delta_{r,s}\alpha_{s}&\text{if }s=v+1.
\end{cases}
\]
\end{Theorem}
We can prove Theorem~\ref{gens2} by the same way as \cite[Theorem~10.22]{U13}.
\begin{Theorem}\label{Tinf}
Suppose that $k+M-n-u_1+q_1\neq0$ and $u_1-u_2+q_1-q_2\geq3$.
The elements \smash{$W^{(1)}_{a,b}$} and \smash{$W^{(2)}_{a,b}$} generate $\mathcal{W}^k(\mathfrak{gl}(M|N),f)$.
\end{Theorem}
The proof of Theorem~\ref{Tinf} can be proven by the same way as \cite[Theorem~3.6]{U4}. We will give the proof in the appendix.

Let us take an integer $1<x<l$ and set
\begin{align*}
&M_1=\sum\limits_{v=1}^x u_v,\qquad N_1=\sum\limits_{v=1}^x  q_v,\qquad
M_2=\sum\limits_{v=x+1}^l  u_v,\qquad N_2=\sum\limits_{v=x+1}^l  q_v.
\end{align*}
We define $f_1$ (resp.\ $f_2$) as a nilpotent element of type $\bigl(1^{u_1-u_2|q_1-q_2},2^{u_2-u_3|q_2-q_3},\dots,p^{u_x-0|q_x-0}\bigr)$ \big(resp.\ $\big(1^{u_{x+1}-u_{x+2}|q_{x+1}-q_{x+2}},2^{u_{x+2}-u_{x+3}|q_{x+2}-q_{x+3}},\dots,p^{u_l-0|q_l-0}\big)$\big) in $\mathfrak{gl}(M_1|N_1)$ (resp.\ $\mathfrak{gl}(M_2|N_2)$ by the same way as $f\in\mathfrak{gl}(M|N)$.

We denote the Miura maps as
\begin{align*}
&\mu_1\colon\ \mathcal{W}^{k+M_2-N_2}(\mathfrak{gl}(M_1|N_1),f_1)\to \bigotimes_{1\leq i\leq x}V^\kappa_i(\mathfrak{gl}(u_i|q_i)),\\
&\mu_2\colon\ \mathcal{W}^{k+M_1-N_1}(\mathfrak{gl}(M_2|N_2),f_2)\to \bigotimes_{x+1\leq i\leq l}V^\kappa_i(\mathfrak{gl}(u_i|q_i)),
\end{align*}
where $\kappa_i$ is an appropriate inner product on $\mathfrak{gl}(u_i|q_i)$.
\begin{Theorem}
Suppose that $k+M-N-u_1+q_1\neq0$.
There exists a homomorphism
\begin{align*}
\Delta_W\colon\ \mathcal{W}^k(\mathfrak{gl}(M|N),f)\to\mathcal{W}^{k+M_2-N_2}(\mathfrak{gl}(M_1|N_1),f_1)\otimes\mathcal{W}^{k+M_1-N_1}(\mathfrak{gl}(M_2|N_2),f_2)
\end{align*}
determined by $\mu=(\mu_1\otimes\mu_2)\circ\Delta_W$.
\end{Theorem}
\begin{proof}
By Theorem~\ref{Tinf}, it is enough to show \smash{$\mu\big(W^{(r)}_{a,b}\big)$} is contained in
\[
\mathcal{W}^{k+M_2-N_2}(\mathfrak{gl}(M_1|N_1),f_1)\otimes\mathcal{W}^{k+M_1-N_1}(\mathfrak{gl}(M_2|N_2),f_2)
\]
for $r=1,2$.
For the latter discussion, we only show the case that $a,b>u_1-u_l$ or $a,b<-q_1+q_l$. By the definition of \smash{$W^{(r)}_{a,b}$}, we have
\begin{gather*}
\mu\big(W^{(1)}_{a,b}\big)=W^{(1)}_{a,b}\otimes1+1\otimes W^{(1)}_{a,b},\\
\mu\big(W^{(2)}_{a,b}\big)=W^{(2)}_{a,b}\otimes 1+1\otimes W^{(2)}_{a,b}-\gamma_{p}\partial W^{(1)}_{a,b}\otimes 1\\
\hphantom{\mu(W^{(2)}_{a,b})=}{}-\sum\limits_{u_1-u_x< u\leq u_1-u_l,-q_1+q_l<u\leq -q_1+q_x}(-1)^{p(e_{a,u})p(e_{b,u})+p(u)}\big(W^{(1)}_{u,b}\big)_{(-1)}W^{(1)}_{a,u}\otimes 1\\
\hphantom{\mu(W^{(2)}_{a,b})=}{}-\sum\limits_{u>u_1-u_l,u<-q_1+q_l}\big((-1)^{p(e_{a,u})p(e_{b,u})+p(u)}W^{(1)}_{u,b}\big)\otimes W^{(1)}_{a,u}\\
\hphantom{\mu(W^{(2)}_{a,b})=}{}+\sum\limits_{u_1-u_{x+1}<u\leq u_1-u_l,-q_1+q_l\leq u<-q_1+q_{x+1}}(-1)^{p(u)}W^{(1)}_{a,u}\otimes W^{(1)}_{u,b}.\tag*{\qed}
\end{gather*}\renewcommand{\qed}{}
\end{proof}

By Theorems~\ref{Main3} and~\ref{Main4}, we obtain
\begin{align*}
&\Phi^1\colon\ Y_{\hbar,\ve+(u_x-q_x-u_l+q_l)\hbar}\big(\widehat{\mathfrak{sl}}(u_x|q_x)\big)\to\mathcal{U}\big(\mathcal{W}^{k+M_2-N_2}(\mathfrak{gl}(M_1|N_1),f_1)\big),\\
&\Phi^2\colon\ Y^{u_{x+1}|q_{x+1},R}_{\hbar,\ve}\big(\widehat{\mathfrak{sl}}(u_l|q_l)\big)\to\mathcal{U}\big(\mathcal{W}^{k+M_1-N_1}(\mathfrak{gl}(M_2|N_2),f_2)\big).
\end{align*}
For a complex number $a\in\mathbb{C}$, we set a homomorphism called the shift operator of the affine super Yangian
\begin{equation*}
\tau_a\colon\ Y_{\hbar,\ve}\bigl(\widehat{\mathfrak{sl}}(m|n)\bigr)\to Y_{\hbar,\ve}\bigl(\widehat{\mathfrak{sl}}(m|n)\bigr)
\end{equation*}
determined by \smash{$X^\pm_{i,0}\mapsto X^\pm_{i,0}$} and \smash{$H_{i,1}\mapsto H_{i,1}+a H_{i,0}$}. Then, by a direct computation, we obtain the compatibility with the coproduct for the affine super Yangian and the parabolic presentation for a $W$-superalgebra.
\begin{Corollary}
We obtain the following relations:
\begin{gather*}
\big(\big(\Phi^1\circ\tau_{(-\gamma_{w}+q_w-q_{w+1})\hbar}\circ\Psi^{u_{x+1}|q_{x+1},u_x|q_x}_2 \big)\otimes\Phi^2\big) \circ\Delta^{m_2|n_2}=\Delta_W\circ\Phi.
\end{gather*}
\end{Corollary}

\appendix

\section{Some formulas for Theorem~\ref{Commutativity}}
In this section, we prepare some formulas for Theorem~\ref{Commutativity}.
By a direct computation, we obtain the following formula.
\begin{Theorem}
For $i\neq j$ and $a,b\geq0$, the following relations hold:
\begin{gather}
\big[E_{i,z}t^{-v-b} E_{z,i}t^{v+b},E_{u,j}t^{-s-a}E_{j,u}t^{s+a}\big] \nonumber\\
\qquad{}=\delta_{u,i}E_{i,z}t^{-v-b} E_{z,j}t^{v+b-s-a}E_{j,u}t^{s+a}\nonumber\\
\quad\qquad{}-(-1)^{p(E_{z,i})p(E_{u,j})}\delta_{z,j}E_{i,z}t^{-v-b} E_{u,i}t^{v+b-s-a}E_{j,u}t^{s+a}\nonumber\\
\quad\qquad{}+(-1)^{p(E_{z,i})p(E_{u,j})}\delta_{u,z}E_{i,j}t^{-v-b-s-a}E_{z,i}t^{v+b}E_{j,u}t^{s+a}\nonumber\\
\quad\qquad{}-(-1)^{p(E_{z,i})p(E_{u,j})}\delta_{u,z}E_{u,j}t^{-s-a}E_{i,z}t^{-v-b} E_{j,i}t^{v+b+s+a}\nonumber\\
\quad\qquad{}+(-1)^{p(E_{z,i})p(E_{u,j})}\delta_{j,z}E_{u,j}t^{-s-a}E_{i,u}t^{-v-b+s+a}E_{z,i}t^{v+b}\nonumber\\
\quad\qquad{} -\delta_{i,u}E_{u,j}t^{-s-a}E_{j,z}t^{-v-b+s+a}E_{z,i}t^{v+b},\label{e1}\\
[E_{i,z}t^{-v-b} E_{z,i}t^{v+b},E_{j,u}t^{-s-a}E_{u,j}t^{s+a}]\nonumber\\
\qquad{}=-(-1)^{p(E_{z,i})p(E_{u,j})}\delta_{u,z}E_{i,z}t^{-v-b}E_{j,i}t^{v+b-s-a}E_{u,j}t^{s+a}\nonumber\\
\quad \qquad{} +(-1)^{p(E_{z,i})p(E_{u,j})}\delta_{j,z}E_{i,u}t^{-v-b-s-a}E_{z,i}t^{v+b}E_{u,j}t^{s+a}\nonumber\\
\quad \qquad{}-\delta_{i,u}E_{j,z}t^{-v-b-s-a}E_{z,i}t^{v+b}E_{u,j}t^{s+a}\nonumber\\
\quad \qquad{} +\delta_{i,u}E_{j,u}t^{-s-a}E_{i,z}t^{-v-b} E_{z,j}t^{v+b+s+a}\nonumber\\
\quad \qquad{} -(-1)^{p(E_{z,i})p(E_{u,j})}\delta_{j,z}E_{j,u}t^{-s-a}E_{i,z}t^{-v-b} E_{u,i}t^{v+b+s+a}\nonumber\\
\quad \qquad{}  +(-1)^{p(E_{z,i})p(E_{u,j})}\delta_{u,z}E_{j,u}t^{-s-a}E_{i,j}t^{s+a-v-b}E_{z,i}t^{v+b},\label{e2}\\
[E_{z,i}t^{-v-b} E_{i,z}t^{v+b},E_{u,j}t^{-s-a}E_{j,u}t^{s+a}]\nonumber\\
\qquad =\delta_{u,z}E_{z,i}t^{-v-b} E_{i,j}t^{v+b-s-a}E_{j,u}t^{s+a}\nonumber\\
\quad \qquad{} +(-1)^{p(E_{z,i})p(E_{u,j})}\delta_{i,u}E_{z,j}t^{-v-b-s-a}E_{i,z}t^{v+b}E_{j,u}t^{s+a}\nonumber\\
\quad \qquad{}-\delta_{j,z}E_{u,i}t^{-v-b-s-a}E_{i,z}t^{v+b}E_{j,u}t^{s+a}\nonumber\\
\quad \qquad{} +\delta_{z,j}E_{u,j}t^{-s-a}E_{z,i}t^{-v-b} E_{i,u}t^{v+b+s+a}\nonumber\\
\quad \qquad{} -(-1)^{p(E_{z,i})p(E_{u,j})}\delta_{i,u}E_{u,j}t^{-s-a}E_{z,i}t^{-v-b} E_{j,z}t^{v+b+s+a}\nonumber\\
\quad \qquad{} -\delta_{z,u}E_{u,j}t^{-s-a}E_{j,i}t^{-v-b+s+a}E_{i,z}t^{v+b}.\label{e3}
\end{gather}
\end{Theorem}
\section{Proof of Theorem~\ref{Commutativity}}
In this appendix, we give a proof of Theorem~\ref{Commutativity}.
Since the affine super Yangian \smash{$Y_{\hbar,\ve}\bigl(\widehat{\mathfrak{sl}}(m|n)\bigr)$} is generated by \smash{$\big\{X^\pm_{i,0}\big\}_{i\in I_{m|n}}$} and \smash{$\widetilde{H}_{1,1}$}, we need to show that
\[
\big\{\Psi_1^{m_1|n_1,m_1+m_2|n_1+n_2}\big(X^\pm_{i,0}\big)\big\}_{i\in I_{m_1|n_1}}\qquad \text{and}\qquad \Psi_1^{m_1|n_1,m_1+m_2|n_1+n_2}\big(\widetilde{H}_{1,1}\big)
\]
 commute with \smash{$\big\{\Psi_2^{m_2|n_2,m_1+m_2|n_1+n_2}\big(X^\pm_{i,0}\big)\big\}_{i\in I_{m_2|n_2}}$} and \smash{$\Psi_2^{m_2|n_2,m_1+m_2|n_1+n_2}\big(\widetilde{H}_{1,1}\big)$}.
The commutatibility with
\[
\Psi_1^{m_1|n_1,m_1+m_2|n_1+n_2}\big(X^\pm_{i,0}\big) \qquad \text{and}\qquad \Psi_2^{m_2|n_2,m_1+m_2|n_1+n_2}\big(X^\pm_{i,0}\big)
\]
 follows from the definitions of two edge contractions.
Thus, it is enough to show the following three relations:
\begin{gather}
\big[\Psi_1^{m_1|n_1,m_1+m_2|n_1+n_2}\big(X^\pm_{i,0}\big), \Psi_2^{m_2|n_2,m_1+m_2|n_1+n_2}\big(\widetilde{H}_{1,1}\big)\big]=0,\label{gather1}\\
\big[\Psi_1^{m_1|n_1,m_1+m_2|n_1+n_2}\big(\widetilde{H}_{1,1}\big), \Psi_2^{m_2|n_2,m_1+m_2|n_1+n_2}\big(X^\pm_{i,0}\big)\big]=0,\label{gather2}\\
\big[\Psi_1^{m_1|n_1,m_1+m_2|n_1+n_2}\big(\widetilde{H}_{1,1}\big), \Psi_2^{m_2|n_2,m_1+m_2|n_1+n_2}\big(\widetilde{H}_{1,1}\big)\big]=0.\label{gather3}
\end{gather}
We will prove \eqref{gather1}--\eqref{gather3} in the following three appendices.

\subsection{The proof of (\ref{gather1})}
This appendix is devoted to the proof of \eqref{gather1}. We only show the $+$ case. The $-$ case can be derived from $+$ case by using the anti-automorphism $\omega$. The cases that $i\neq 0, m_1$ can be proven by a direct computation. We only show the case that $i=0$ and $i=m$.

First, we show the case that $i=m_1$. By the definition of two edge contractions, we have
\begin{align}
&\big[\Psi_1^{m_1|n_1,m_1+m_2|n_1+n_2}\big(X^+_{m_1,0}\big),\Psi_2^{m_2|n_2,m_1+m_2|n_1+n_2}\big(\widetilde{H}_{1,1}\big)\big]\nonumber\\
&\qquad=\big[E_{m_1,-1},\widetilde{H}_{1+m_1,1}\big]+[E_{m_1,-1},R_1-R_2]+[E_{m_1,-1},S_1-S_2].\label{above}
\end{align}
We will compute each terms of the right-hand side of \eqref{above}. In order to simplify the notation, here after, we denote the $i$-th term of the right-hand side of the equation $(\cdot)$ by $(\cdot)_i$. By a direct computation, we obtain
\begin{align}
&\eqref{above}_2=\hbar\sum\limits_{v\geq0}  E_{m_1,1+m_1}t^{-v}E_{1+m_1,-1}t^{v}-\hbar\sum\limits_{v\geq0}  E_{m_1,2+m_1}t^{-v}E_{2+m_1,-1}t^{v},\label{ER1}\\
&\eqref{above}_3=-\hbar\sum\limits_{v\geq0}  E_{m_1,1+m_1}t^{-v-1}E_{1+m_1,-1}t^{v+1}+\hbar\sum\limits_{v\geq0}  E_{m_1,2+m_1}t^{-v-1}E_{2+m_1,-1}t^{v+1}.\label{ES1}
\end{align}
We can rewrite{\samepage
\begin{align}
\eqref{above}_1&=\big[[E_{m_1,3+m_1},E_{3+m_1,-1}],\widetilde{H}_{1+m_1,1}\big]\nonumber\\
&=\big[\big[E_{m_1,3+m_1},\widetilde{H}_{1+m_1,1}\big],E_{3+m_1,-1}\big]+\big[E_{m_1,3+m_1},\big[E_{3+m_1,-1},\widetilde{H}_{1+m_1,1}\big]\big].\label{EE}
\end{align}}%
Since \smash{$E_{3+m_1,-1}=\big(\prod_{i=3+m_1}^{m_1+m_2-1} \ad\big(X^+_{i,0}\big)\big)X^+_{m_1+m_2,0}$} holds, we obtain
\smash{$\big[E_{3+m_1,-1},\widetilde{H}_{1+m_1,1}\big]=0$}
by~\eqref{Eq2.6}. Thus, $\eqref{EE}_2$ is equal to zero.
Since \smash{$E_{m_1,3+m_1}=\big[X^+_{m_1,0},\big[X^+_{1+m_1,0},X^+_{2+m_1,0}\big]\big]$} holds, we~have
\begin{align*}
\big[E_{m_1,3+m_1},\widetilde{H}_{1+m_1,1}\big] ={}&\big[X^+_{m_1,1},\big[X^+_{1+m_1,0},X^+_{2+m_1,0}\big]\big]\\
&{}{-}\, 2\big[X^+_{m_1,0}, \big[X^+_{1+m_1,1},X^+_{2+m_1,0}\big]\big]+\big[X^+_{m_1,0},\big[X^+_{1+m_1,0},X^+_{2+m_1,1}\big]\big].
\end{align*}
By \eqref{Eq2.8} and \eqref{Eq2.13}, we have
\begin{align}
&\bigl[X^+_{m_1,1},\bigl[X^+_{1+m_1,0},X^+_{2+m_1,0}\bigr]\bigr]-\bigl[X^+_{m_1,0},\bigl[X^+_{1+m_1,1},X^+_{2+m_1,0}\bigr]\bigr]\nonumber\\
&\qquad=-\frac{\hbar}{2}\big\{X^+_{m_1,0},\bigl[X^+_{1+m_1,0},X^+_{2+m_1,0}\bigr]\big\}=-\hbar E_{m_1,m_1+1}E_{m_1+1,m_1+3}+\frac{\hbar}{2}E_{m_1,m_1+3},\label{45}\\
&-\bigl[X^+_{m_1,0},\bigl[X^+_{1+m_1,1},X^+_{2+m_1,0}\bigr]\bigr]+\bigl[X^+_{m_1,0},\bigl[X^+_{1+m_1,0},X^+_{2+m_1,1}\bigr]\bigr]\nonumber\\
&\qquad=\frac{\hbar}{2}\big\{\bigl[X^+_{m_1,0},X^+_{1+m_1,0}\bigr],X^+_{2+m_1,0}\big\}=\hbar E_{m_1,m_1+2}E_{m_1+2,m_1+3}-\frac{\hbar}{2}E_{m_1,m_1+3}.\label{46}
\end{align}
By \eqref{45} and \eqref{46}, we have
\begin{align}
\eqref{EE}_1&=-\hbar E_{m_1,m_1+1}E_{m_1+1,-1}+\hbar E_{m_1,m_1+2}E_{m_1+2,-1}.\label{47-1}
\end{align}
By adding \eqref{ER1}, \eqref{ES1} and \eqref{47-1}, we find that \eqref{above} is equal to zero.

Next, we show the case that $i=0$.
By the definition of two edge contractions, we have
\begin{align}
&\big[\Psi_1^{m_1|n_1,m_1+m_2|n_1+n_2}\big(X^+_{0,0}\big),\Psi_2^{m_2|n_2,m_1+m_2|n_1+n_2}\big(\widetilde{H}_{1,1}\big)\big]\nonumber\\
&\qquad=\big[E_{-n_1,1}t,\widetilde{H}_{1+m_1,1}\big]+[E_{-n_1,1}t,R_1-R_2]+[E_{-n_1,1}t,S_1-S_2].\label{above2}
\end{align}
By a direct computation, we obtain
\begin{align}
\eqref{above2}_2&=-\hbar\sum\limits_{v\geq0}  E_{-n_1,1+m_1}t^{-v}E_{1+m_1,1}t^{v+1}+\hbar\sum\limits_{v\geq0}  E_{-n_1,2+m_1}t^{-v}E_{2+m_1,1}t^{v+1},\label{ER2}\\
\eqref{above2}_3&=\hbar\sum\limits_{v\geq0}  E_{-n_1,1+m_1}t^{-v} E_{1+m_1,1}t^{v+1}-\hbar\sum\limits_{v\geq0}  E_{-n_1,2+m_1}t^{-v} E_{2+m_1,1}t^{v+1}.\label{ES2}
\end{align}
Since we obtain \smash{$E_{-n_1,1}t=\big(\prod_{i=-n_1}^{-n_1-n_2+1}\big)\ad\big(X^+_{i,0}\big)X^+_{-n_1-n_2,0}$}, we find that~$\eqref{above2}_1$ is equal to zero by~\eqref{Eq2.13}. By adding \eqref{ER2} and \eqref{ES2}, we obtain $\eqref{above2}=0$.

\subsection{The proof of (\ref{gather2})}
In this appendix, we will show the relation~\eqref{gather2} in the same way as~\eqref{gather1}. Since the $-$ case can be derived from $+$ case by using the anti-automorphism $\omega$, we will only show the $+$ case. Moreover, we only show the case that $i=0$. The other cases can be proven in a similar way.
By the definition of \smash{$\Psi_1^{m_1|n_1,m_1+m_2|n_1+n_2}$} and \smash{$\Psi_2^{m_2|n_2,m_1+m_2|n_1+n_2}$}, we have
\begin{align}
&\big[\Psi_1^{m_1|n_1,m_1+m_2|n_1+n_2}\big(\widetilde{H}_{1,1}\big),\Psi_2^{m_2|n_2,m_1+m_2|n_1+n_2}\big(X^+_{0,0}\big)\big)\big]\label{above3}\\
&\qquad=[\widetilde{H}_{1,1},E_{-n_1-n_2,m_1+1}t]-[P_1-P_2,E_{-n_1-n_2,m_1+1}t]+[Q_1-Q_2,E_{-n_1-n_2,m_1+1}t].\nonumber
\end{align}
By a direct computation, we obtain
\begin{align}
\eqref{above2}_2&=\hbar\sum\limits_{v\geq0}   E_{1,m_1+1}t^{-v-1} E_{-n_1-n_2,1}t^{v+2}-\hbar\sum\limits_{v\geq0}   E_{2,m_1+1}t^{-v-1} E_{-n_1-n_2,2}t^{v+2},\label{EP1}\\
\eqref{above2}_3&=-\hbar\sum\limits_{v\geq0}   E_{1,m_1+1}t^{-v} E_{-n_1-n_2,1}t^{v+1}+\hbar\sum\limits_{v\geq0}   E_{2,m_1+1}t^{-v} E_{-n_1-n_2,2}t^{v+1}.\label{EQ1}
\end{align}
Since $E_{-n_1-n_2,m_1+1}t=[E_{-n_1-n_2,3}t,E_{3,m_1+1}]$ holds, we have
\begin{align}
\big[\widetilde{H}_{1,1},E_{-n_1-n_2,m_1+1}t\big]&=\big[\widetilde{H}_{1,1},[E_{-n_1-n_2,3}t,E_{3,m_1+1}]\big]\label{EEE}\\
&=\big[\big[\widetilde{H}_{1,1},E_{-n_1-n_2,3}t\big],E_{3,m_1+1}\big]+\big[E_{-n_1-n_2,3}t,[\widetilde{H}_{1,1},E_{3,m_1+1}\big]\big].\nonumber
\end{align}
Since \smash{$E_{3,m_1+1}=\big(\prod_{i=3}^{m_1-1}\ad\big(X^+_{i,0}\big)\big)X^+_{m_1,0}$} holds by a direct computation, we obtain
\begin{align*}
\eqref{EEE}_2&=[E_{-n_1-n_2,3}t,0]=0
\end{align*}
by \eqref{Eq2.13}. Since we obtain
\smash{$E_{-n_1-n_2,3}t=\big[X^+_{0,0},\big[X^+_{1,0},X^+_{2,0}\big]\big]$}, we have
\begin{align*}
&\big[\widetilde{H}_{1,1},\big[X^+_{0,0},\big[X^+_{1,0},X^+_{2,0}\big]\big]\big]\\
&\qquad=-\big[X^+_{0,1},\big[X^+_{1,0},X^+_{2,0}\big]\big]+2\big[X^+_{0,0},\big[X^+_{1,1},X^+_{2,0}\big]\big]-\big[X^+_{0,0},\big[X^+_{1,0},X^+_{2,1}\big]\big]
\end{align*}
by \eqref{Eq2.6}. By \eqref{Eq2.8} and \eqref{Eq2.13}, we obtain
\begin{align}
&-\big[X^+_{0,1},\big[X^+_{1,0},X^+_{2,0}\big]\big]+\big[X^+_{0,0},\big[X^+_{1,1},X^+_{2,0}\big]\big]\nonumber\\
&\qquad=\frac{\hbar}{2}\big\{X^+_{0,0},\big[X^+_{1,0},X^+_{2,0}\big]\big\}=\hbar E_{1,3}E_{-n_1-n_2,1}t+\frac{\hbar}{2}E_{-n_1-n_2,3}t,\label{47}\\
&\big[X^+_{0,0},\big[X^+_{1,1},X^+_{2,0}\big]\big]-\big[X^+_{0,0},\big[X^+_{1,0},X^+_{2,1}\big]\big]\nonumber\\
&\qquad=-\frac{\hbar}{2}\big\{\big[X^+_{0,0},X^+_{1,0}\big],X^+_{2,0}\big\}=-\hbar E_{2,3}E_{-n_1-n_2,2}t-\frac{\hbar}{2}E_{-n_1-n_2,3}t.\label{48}
\end{align}
By adding \eqref{47} and \eqref{48}, we obtain
\begin{align}
\eqref{EEE}_1&=\hbar E_{1,m_1+1}E_{-n_1-n_2,1}t-\hbar E_{2,m_1+1}E_{-n_1-n_2,2}t.\label{48-1}
\end{align}
By adding \eqref{EP1}, \eqref{EQ1} and \eqref{48-1}, we find that \eqref{above3} is equal to zero.
\subsection{The proof of (\ref{gather3})}
This appendix is devoted to the proof of \eqref{gather3}. Similarly to \cite[Section~3]{U2}, we define the elements of \smash{$\widetilde{Y}_{\hbar,\ve}\bigl(\widehat{\mathfrak{sl}}(m|n)\bigr)$}:
\begin{align*}
J(h_i)&=\widetilde{H}_{i,1}+A_i-A_{i+1},\
J\big(x^\pm_i\big)=\begin{cases}
\mp(-1)^{p(i)}\big[J(h_{i-1}),x^\pm_i\big]&\text{if }i\neq 0,\\
\mp\big[J(h_1),x^\pm_0\big]&\text{if }i=0,
\end{cases}
\end{align*}
where
\begin{align*}
A_i
={}&\frac{\hbar}{2}\sum\limits_{\substack{s\geq0\\i<u\leq m+n}}  E_{u,i}t^{-s}E_{i,u}t^s-\frac{\hbar}{2}(-1)^{p(i)}\sum\limits_{\substack{s\geq0\\1\leq u<i}}  (-1)^{p(u)}E_{i,u}t^{-s}E_{u,i}t^s\\
&{}{+}\,\frac{\hbar}{2}\sum\limits_{\substack{s\geq0\\1\leq u<i}}  E_{u,i}t^{-s-1}E_{i,u}t^{s+1}-\frac{\hbar}{2}(-1)^{p(i)}\sum\limits_{\substack{s\geq0\\i<u\leq m+n}} (-1)^{p(u)} E_{i,u}t^{-s-1}E_{u,i}t^{s+1}.
\end{align*}
For the simplicity, we sometimes denote
\begin{alignat*}{3}
&A_{i,1}=\frac{\hbar}{2}\sum\limits_{\substack{s\geq0\\i<u\leq m+n}}  E_{u,i}t^{-s}E_{i,u}t^s,&&\qquad A_{i,2}=\frac{\hbar}{2}(-1)^{p(i)}\sum\limits_{\substack{s\geq0\\1\leq u<i}}  (-1)^{p(u)}E_{i,u}t^{-s}E_{u,i}t^s&\\
&A_{i,3}=\frac{\hbar}{2}\sum\limits_{\substack{s\geq0\\1\leq u<i}}  E_{u,i}t^{-s-1}E_{i,u}t^{s+1},&&\qquad A_{i,4}=\frac{\hbar}{2}(-1)^{p(i)}\sum\limits_{\substack{s\geq0\\i<u\leq m+n}} (-1)^{p(u)} E_{i,u}t^{-s-1}E_{u,i}t^{s+1}.&
\end{alignat*}

Let $\alpha$ be a positive real root of \smash{$\widehat{\mathfrak{sl}}(m|n)$}. We take $x^\pm_\alpha$ be a non-zero element of the root space with $\pm\alpha$. We also take simple roots $\{\alpha_i\}_{i\in I_{m|n}}$ of \smash{$\widehat{\mathfrak{sl}}(m|n)$}.
\begin{Lemma}[{\cite[Proposition 4.26]{U2}}]\label{J}
There exists a complex number $c_{\alpha,i}$ satisfying that
\begin{equation*}
(\alpha_j,\alpha)\big[J(h_i),x^\pm_\alpha\big]-(\alpha_i,\alpha)\big[J(h_j),x^\pm_\alpha\big]=\pm c_{\alpha,i}x_\alpha^\pm,
\end{equation*}
where $(\, ,\, )$ is defined by $(\alpha_i,\alpha_j)=a_{i,j}$.
\end{Lemma}
By Lemma~\ref{J} and the definition of $P_i$, $Q_i$, $R_j$, $S_j$, it is enough to show the relation
\begin{align}
[P_i-Q_i,R_j+S_j]+[A_i,R_j+S_j]-[P_i-Q_i,A_{m_1+j}]&=0\label{conclusion}
\end{align}
for $i,j=1,2$. By a direct computation, we obtain
\begin{gather}
[P_i,R_{j}]
={-}\hbar^2\sum\limits_{s,v\geq0}  \sum\limits_{u=-n_1}^{-1}E_{i,j+m_1}t^{-v-1} E_{u,i}t^{v+1-s}E_{j+m_1,u}t^{s}\nonumber\\
\hphantom{[P_i,R_{j}]=}{}\hspace{0.2mm}
+\hbar^2\sum\limits_{s,v\geq0}  \sum\limits_{u=-n_1}^{-1}E_{u,j+m_1}t^{-s}E_{i,u}t^{-v-1+s}E_{j+m_1,i}t^{v+1},\label{PR}\\
[P_i,S_{j}]
=\hbar^2\sum\limits_{s,v\geq0}  \sum\limits_{z=m_1+1}^{m_1+m_2}  E_{i,z}t^{-v-1} E_{z,j+m_1}t^{v-s}E_{j+m_1,i}t^{s+1}\nonumber\\
\hphantom{[P_i,S_{j}]=}{}\hspace{0.2mm}
-\hbar^2\sum\limits_{s,v\geq0}  \sum\limits_{u=1}^{m_1} E_{i,j+m_1}t^{-v-1} E_{u,i}t^{v-s}E_{j+m_1,u}t^{s+1}\nonumber\\
\hphantom{[P_i,S_{j}]=}{}\hspace{0.2mm}
+\hbar^2\sum\limits_{s,v\geq0}  \sum\limits_{u=1}^{m_1} E_{u,j+m_1}t^{-s-1}E_{i,u}t^{-v+s}E_{j+m_1,i}t^{v+1}\nonumber\\
\hphantom{[P_i,S_{j}]=}{}\hspace{0.2mm}
-\hbar^2\sum\limits_{s,v\geq0}  \sum\limits_{z=m_1+1}^{m_1+m_2}  E_{i,j+m_1}t^{-s-1}E_{j+m_1,z}t^{-v+s}E_{z,i}t^{v+1},\label{PS}\\
[Q_i,R_{j}]=0,\nonumber\\
[Q_i,S_{j}]
=\hbar^2\sum\limits_{s,v\geq0}  \sum\limits_{z=-n_1-n_2}^{-n_1-1}  E_{i,z}t^{-v-1} E_{z,j+m_1}t^{v-s}E_{j+m_1,i}t^{s+1}\nonumber\\
\hphantom{[Q_i,S_{j}]=}{}\hspace{0.2mm}
-\hbar^2\sum\limits_{s,v\geq0}  \sum\limits_{z=-n_1-n_2}^{-n_1-1}  E_{i,j+m_1}t^{-s-1}E_{j+m_1,z}t^{-v+s}E_{z,i}t^{v+1}\label{QS}
\end{gather}
By the definition of $A_i$ and \eqref{e1}, we can rewrite $[P_i,A_{j+m_1}]$ as
\begin{align*}
&[P_i,A_{j+m_1,1}]-[P_i,A_{j+m_1,2}]+[P_i,A_{j+m_1,3}]-[P_i,A_{j+m_1,4}].
\end{align*}
By \eqref{e1}, we obtain
\begin{align}
[P_i,A_{j+m_1,1}]
={}&{-}\frac{\hbar^2}{2}\sum\limits_{\substack{s,v\geq0\\u>j+m_1}}  E_{i,j+m_1}t^{-v-1} E_{u,i}t^{v+1-s}E_{j+m_1,u}t^{s}\nonumber\\
&{+}\,\frac{\hbar^2}{2}\sum\limits_{\substack{s,v\geq0}} \sum\limits_{z=j+m_1+1}^{m_1+m_2}  E_{i,j+m_1}t^{-v-1-s}E_{z,i}t^{v+1}E_{j+m_1,z}t^{s}\nonumber\\
&{-}\,\frac{\hbar^2}{2}\sum\limits_{\substack{s,v\geq0}} \sum\limits_{z=j+m_1+1}^{m_1+m_2}  E_{z,j+m_1}t^{-s}E_{i,z}t^{-v-1} E_{j+m_1,i}t^{v+1+s}\nonumber\\
&{+}\,\frac{\hbar^2}{2}\sum\limits_{\substack{s,v\geq0\\u>j+m_1}}  E_{u,j+m_1}t^{-s}E_{i,u}t^{-v-1+s}E_{j+m_1,i}t^{v+1}.\label{PA1-1}
\end{align}
By a direct computation, we obtain
\begin{gather*}
\eqref{PA1-1}_1+\eqref{PA1-1}_2={-}\frac{\hbar^2}{2}\sum\limits_{\substack{s,v\geq0}} \sum\limits_{u=-n_1-n_2}^{-1} E_{i,j+m_1}t^{-v-1} E_{u,i}t^{v+1-s}E_{j+m_1,u}t^{s}\nonumber\\
\hphantom{\eqref{PA1-1}_1+\eqref{PA1-1}_2=}{}\hspace{0.2mm}
-\frac{\hbar^2}{2}\sum\limits_{\substack{s,v\geq0}} \sum\limits_{z=j+m_1+1}^{m_1+m_2}  E_{i,j+m_1}t^{-v-1}E_{z,i}t^{-s}E_{j+m_1,z}t^{s+v+1},\nonumber\\
\eqref{PA1-1}_3+\eqref{PA1-1}_4=\frac{\hbar^2}{2}\sum\limits_{\substack{s,v\geq0}} \sum\limits_{z=j+m_1+1}^{m_1+m_2}  E_{z,j+m_1}t^{-s-v-1}E_{i,z}t^{s} E_{j+m_1,i}t^{v+1}\nonumber\\
\hphantom{\eqref{PA1-1}_3+\eqref{PA1-1}_4=}{}\hspace{0.2mm}
+\frac{\hbar^2}{2}\sum\limits_{\substack{s,v\geq0}} \sum\limits_{u=-n_1-n_2}^{-1} E_{u,j+m_1}t^{-s}E_{i,u}t^{-v-1+s}E_{j+m_1,i}t^{v+1}.
\end{gather*}
Then, we can rewrite
\begin{align}
[P_i,A_{j+m_1,1}]
={}&{-}\frac{\hbar^2}{2}\sum\limits_{\substack{s,v\geq0}} \sum\limits_{u=-n_1-n_2}^{-n_1-1} E_{i,j+m_1}t^{-s-v-1} E_{u,i}t^{v+1}E_{j+m_1,u}t^{s}\nonumber\\
&{-}\,\frac{\hbar^2}{2}\sum\limits_{\substack{s,v\geq0}} \sum\limits_{u=-n_1-n_2}^{-n_1-1} E_{i,j+m_1}t^{-v-1} E_{u,i}t^{-s}E_{j+m_1,u}t^{s+v+2}\nonumber\\
&{-}\,\frac{\hbar^2}{2}\sum\limits_{\substack{s,v\geq0}} \sum\limits_{u=-n_1}^{-1} E_{i,j+m_1}t^{-v-1} E_{u,i}t^{v+1-s}E_{j+m_1,u}t^{s}\nonumber\\
&{-}\,\frac{\hbar^2}{2}\sum\limits_{\substack{s,v\geq0}} \sum\limits_{z=j+m_1+1}^{m_1+m_2}  E_{i,j+m_1}t^{-v-1}E_{z,i}t^{-s}E_{j+m_1,z}t^{s+v+1}\nonumber\\
&{+}\,\frac{\hbar^2}{2}\sum\limits_{\substack{s,v\geq0}} \sum\limits_{z=j+m_1+1}^{m_1+m_2}  E_{z,j+m_1}t^{-s-v-1}E_{i,z}t^{s} E_{j+m_1,i}t^{v+1}\nonumber\\
&{+}\,\frac{\hbar^2}{2}\sum\limits_{\substack{s,v\geq0}} \sum\limits_{u=-n_1-n_2}^{-n_1-1} E_{u,j+m_1}t^{-s-v-1}E_{i,u}t^{s}E_{j+m_1,i}t^{v+1}\nonumber\\
&{+}\,\frac{\hbar^2}{2}\sum\limits_{\substack{s,v\geq0}} \sum\limits_{u=-n_1-n_2}^{-n_1-1} E_{u,j+m_1}t^{-s}E_{i,u}t^{-v-1}E_{j+m_1,i}t^{s+v+1}\nonumber\\
&{+}\,\frac{\hbar^2}{2}\sum\limits_{\substack{s,v\geq0}} \sum\limits_{u=-n_1}^{-1} E_{u,j+m_1}t^{-s}E_{i,u}t^{-v-1+s}E_{j+m_1,i}t^{v+1}.\label{PA1}
\end{align}
Similarly, by \eqref{e1} and \eqref{e2}, we can rewrite
\begin{gather}
[P_i,A_{j+m_1,2}]
=\frac{\hbar^2}{2}\sum\limits_{\substack{s,v\geq0}} \sum\limits_{u=1}^{m_1} E_{i,u}t^{-v-1-s}E_{j+m_1,i}t^{v+1}E_{u,j+m_1}t^{s}\nonumber\\
\hphantom{[P_i,A_{j+m_1,2}]=}{}\hspace{0.2mm} -\frac{\hbar^2}{2}\sum\limits_{\substack{s,v\geq0}} \sum\limits_{z=j+m_1}^{m_1+m_2}  E_{j+m_1,z}t^{-v-1-s}E_{i,j+m_1}t^{s}E_{z,i}t^{v+1}\nonumber\\
\hphantom{[P_i,A_{j+m_1,2}]=}{}\hspace{0.2mm} +\frac{\hbar^2}{2}\sum\limits_{\substack{s,v\geq0}} \sum\limits_{z=j+m_1}^{m_1+m_2}  E_{i,z}t^{-v-1}E_{j+m_1,i}t^{-s} E_{z,j+m_1}t^{v+1+s}\nonumber\\
\hphantom{[P_i,A_{j+m_1,2}]=}{}\hspace{0.2mm} -\frac{\hbar^2}{2}\sum\limits_{\substack{s,v\geq0}} \sum\limits_{u=1}^{m_1} E_{j+m_1,u}t^{-s}E_{i,j+m_1}t^{-v-1} E_{u,i}t^{v+1+s},\label{PA2}\\
[P_i,A_{j+m_1,3}]
=\frac{\hbar^2}{2}\sum\limits_{\substack{s,v\geq0}} \sum\limits_{z=m_1+1}^{m_1+m_2}  E_{i,z}t^{-v-1} E_{z,j+m_1}t^{v-s}E_{j+m_1,i}t^{s+1}\nonumber\\
\hphantom{[P_i,A_{j+m_1,3}]=}{}\hspace{0.2mm} -\frac{\hbar^2}{2}\sum\limits_{\substack{s,v\geq0}} \sum\limits_{u=1}^{m_1}E_{i,j+m_1}t^{-v-1} E_{u,i}t^{v-s}E_{j+m_1,u}t^{s+1}\nonumber\\
\hphantom{[P_i,A_{j+m_1,3}]=}{}\hspace{0.2mm} -\frac{\hbar^2}{2}\sum\limits_{\substack{s,v\geq0}} \sum\limits_{z=m_1+1}^{j+m_1-1}  E_{i,j+m_1}t^{-v-1}E_{u,i}t^{-s}E_{j+m_1,u}t^{v+s+1}\nonumber\\
\hphantom{[P_i,A_{j+m_1,3}]=}{}\hspace{0.2mm} +\frac{\hbar^2}{2}\sum\limits_{\substack{s,v\geq0}} \sum\limits_{z=m_1+1}^{j+m_1-1}  E_{z,j+m_1}t^{-s-v-1}E_{i,z}t^{s} E_{j+m_1,i}t^{v+1}\nonumber\\
\hphantom{[P_i,A_{j+m_1,3}]=}{}\hspace{0.2mm} +\frac{\hbar^2}{2}\sum\limits_{\substack{s,v\geq0}} \sum\limits_{u=1}^{m_1}E_{u,j+m_1}t^{-s-1}E_{i,u}t^{-v+s}E_{j+m_1,i}t^{v+1}\nonumber\\
\hphantom{[P_i,A_{j+m_1,3}]=}{}\hspace{0.2mm} -\frac{\hbar^2}{2}\sum\limits_{\substack{s,v\geq0}} \sum\limits_{z=m_1+1}^{m_1+m_2}  E_{i,j+m_1}t^{-s-1}E_{j+m_1,z}t^{-v+s}E_{z,i}t^{v+1},\label{PA3}\\
[P_i,A_{j+m_1,4}]
={-}\frac{\hbar^2}{2}\sum\limits_{\substack{s,v\geq0}} \sum\limits_{z=j+m_1+1}^{m_1+m_2}  E_{i,z}t^{-v-1}E_{j+m_1,i}t^{-s}E_{z,j+m_1}t^{s+v+1}\nonumber\\
\hphantom{[P_i,A_{j+m_1,4}]=}{}\hspace{0.2mm} -\frac{\hbar^2}{2}\sum\limits_{\substack{s,v\geq0}} \sum\limits_{u=-n_1-n_2}^{-n_1-1} E_{i,u}t^{-v-s-2}E_{j+m_1,i}t^{v+1}E_{u,j+m_1}t^{s+1}\nonumber\\
\hphantom{[P_i,A_{j+m_1,4}]=}{}\hspace{0.2mm} -\frac{\hbar^2}{2}\sum\limits_{\substack{s,v\geq0}} \sum\limits_{u=-n_1}^{-1} E_{i,u}t^{-v-s-2}E_{j+m_1,i}t^{v+1}E_{u,j+m_1}t^{s+1}\nonumber\\
\hphantom{[P_i,A_{j+m_1,4}]=}{}\hspace{0.2mm} +\frac{\hbar^2}{2}\sum\limits_{\substack{s,v\geq0}} \sum\limits_{u=-n_1-n_2}^{-n_1-1} E_{j+m_1,u}t^{-s-1}E_{i,j+m_1}t^{-v-1} E_{u,i}t^{v+s+2}\nonumber\\
\hphantom{[P_i,A_{j+m_1,4}]=}{}\hspace{0.2mm} +\frac{\hbar^2}{2}\sum\limits_{\substack{s,v\geq0}} \sum\limits_{u=-n_1}^{-1} E_{j+m_1,u}t^{-s-1}E_{i,j+m_1}t^{-v-1} E_{u,i}t^{v+s+2}\nonumber\\
\hphantom{[P_i,A_{j+m_1,4}]=}{}\hspace{0.2mm} +\frac{\hbar^2}{2}\sum\limits_{\substack{s,v\geq0}} \sum\limits_{z=j+m_1+1}^{m_1+m_2}  E_{j+m_1,z}t^{-s-v-1}E_{i,j+m_1}t^{s}E_{z,i}t^{v+1}.\label{PA4}
\end{gather}
By a direct computation, we have
\begin{align*}
&\eqref{PA1}_2+\eqref{PA3}_3=-\frac{\hbar^2}{2}\sum\limits_{\substack{s,v\geq0}}  \sum\limits_{\substack{u=m_1+1\\u\neq j+m_1}}^{m_1+m_2}  E_{u,i}t^{-s}E_{i,j+m_1}t^{-v-1}E_{j+m_1,u}t^{v+s+1},\nonumber\\
&\eqref{PA1}_3+\eqref{PA3}_4=\frac{\hbar^2}{2}\sum\limits_{\substack{s,v\geq0}} \sum\limits_{\substack{z=m_1+1\\z\neq j+m_1}}^{m_1+m_2}  E_{z,j+m_1}t^{-s-v-1} E_{j+m_1,i}t^{v+1}E_{i,z}t^{s},\nonumber\\
&-\eqref{PA2}_2-\eqref{PA4}_6=\frac{\hbar^2}{2}\sum\limits_{\substack{s,v\geq0}}  E_{j+m_1,j+m_1}t^{-v-1-s}E_{i,j+m_1}t^{s}E_{j+m_1,i}t^{v+1},\\
&-\eqref{PA2}_3-\eqref{PA4}_1=-\frac{\hbar^2}{2}\sum\limits_{\substack{s,v\geq0}}  E_{i,j+m_1}t^{-v-1}E_{i,i}t^{-s}E_{j+m_1,i}t^{s+v+1}.
\end{align*}
Then, we find that
\begin{gather}
\eqref{PA1}_2+\eqref{PA3}_3+\eqref{PA1}_3+\eqref{PA3}_4-\eqref{PA2}_2-\eqref{PA4}_6-\eqref{PA2}_3-\eqref{PA4}_1\nonumber\\
\qquad=- \frac{\hbar^2}{2}\sum\limits_{\substack{s,v\geq0}}  \sum\limits_{\substack{u=m_1+1}}^{m_1+m_2}  E_{u,i}t^{-s}E_{i,j+m_1}t^{-v-1}E_{j+m_1,u}t^{v+s+1}\nonumber\\
\hspace{12.05mm}+\frac{\hbar^2}{2}\sum\limits_{\substack{s,v\geq0}} \sum\limits_{\substack{z=m_1+1\\z\neq j+m_1}}^{m_1+m_2}  E_{z,j+m_1}t^{-s-v-1} E_{j+m_1,i}t^{v+1}E_{i,z}t^{s}.\label{PA}
\end{gather}
Similarly, by the definition of $A_i$, we can rewrite $[Q_i,A_{j+m_1}]$ as
\begin{align*}
[Q_i,A_{j+m_1,1}]-[Q_i,A_{j+m_1,2}]+[Q_i,A_{j+m_1,3}]-[Q_i,A_{j+m_1,4}].
\end{align*}
By \eqref{e1} and \eqref{e2}, we obtain
\begin{gather}
[Q_i,A_{j+m_1,1}]
=\frac{\hbar^2}{2}\sum\limits_{\substack{s,v\geq0}} \sum\limits_{z=-n_1-n_2}^{-n_1-1}  E_{i,j+m_1}t^{-v-1-s}E_{z,i}t^{v+1}E_{j+m_1,z}t^{s}\nonumber\\
\hphantom{[Q_i,A_{j+m_1,1}]=}{}\hspace{0.2mm}
-\frac{\hbar^2}{2}\sum\limits_{\substack{s,v\geq0}} \sum\limits_{z=-n_1-n_2}^{-n_1-1}  E_{z,j+m_1}t^{-s}E_{i,z}t^{-v-1}E_{j+m_1,i}t^{v+1+s},\label{QA1}\\
[Q_i,A_{j+m_1,2}]
=-\frac{\hbar^2}{2}\sum\limits_{\substack{s,v\geq0}} \sum\limits_{z=-n_1-n_2}^{-n_1-1}  E_{j+m_1,z}t^{-v-s-1}E_{i,j+m_1}t^{s}E_{z,i}t^{v+1}\nonumber\\
\hphantom{[Q_i,A_{j+m_1,2}]=}{}\hspace{0.2mm}
+\frac{\hbar^2}{2}\sum\limits_{\substack{s,v\geq0}} \sum\limits_{z=-n_1-n_2}^{-n_1-1}  E_{i,z}t^{-v-1}E_{j+m_1,i}t^{-s} E_{z,j+m_1}t^{v+1+s},\label{QA2}\\
[Q_i,A_{j+m_1,3}]
=\frac{\hbar^2}{2}\sum\limits_{\substack{s,v\geq0}} \sum\limits_{z=-n_1-n_2}^{-n_1-1}  E_{i,z}t^{-v-1} E_{z,j+m_1}t^{v-s}E_{j+m_1,i}t^{s+1}\nonumber\\
\hphantom{[Q_i,A_{j+m_1,3}]=}{}\hspace{0.2mm}
-\frac{\hbar^2}{2}\sum\limits_{\substack{s,v\geq0}} \sum\limits_{z=-n_1-n_2}^{-n_1-1}  E_{i,j+m_1}t^{-s-1}E_{j+m_1,z}t^{s-v}E_{z,i}t^{v+1},\label{QA3}\\
[Q_i,A_{j+m_1,4}]
=-\frac{\hbar^2}{2}\sum\limits_{\substack{s,v\geq0}} \sum\limits_{z=-n_1-n_2}^{-n_1-1}  E_{i,z}t^{-v-1}E_{j+m_1,i}t^{v-s}E_{z,j+m_1}t^{s+1}\nonumber\\
\hphantom{[Q_i,A_{j+m_1,4}]=}{}\hspace{0.2mm}
+\frac{\hbar^2}{2}\sum\limits_{\substack{s,v\geq0}} \sum\limits_{z=-n_1-n_2}^{-n_1-1}  E_{j+m_1,z}t^{-s-1}E_{i,j+m_1}t^{s-v}E_{z,i}t^{v+1}.\label{QA4}
\end{gather}
By a direct computation, we obtain
\begin{align}
&-\eqref{QA2}_1-\eqref{QA4}_2=-\frac{\hbar^2}{2}\sum\limits_{\substack{s,v\geq0}} \sum\limits_{z=-n_1-n_2}^{-n_1-1}  E_{j+m_1,z}t^{-s-1}E_{i,j+m_1}t^{-v-1}E_{z,i}t^{s+v+2},\label{QA5}\\
&-\eqref{QA2}_2-\eqref{QA4}_1=\frac{\hbar^2}{2}\sum\limits_{\substack{s,v\geq0}} \sum\limits_{z=-n_1-n_2}^{-n_1-1}  E_{i,z}t^{-s-v-2}E_{j+m_1,i}t^{v+1}E_{z,j+m_1}t^{s+1}.\label{QA6}
\end{align}
By the definition of $A_i$, we can rewrite $[A_i,R_{j+m_1}]$ as
\begin{align*}
[A_{i,1},R_{j+m_1}]-[A_{i,2},R_{j+m_1}]+[A_{i,3},R_{j+m_1}]-[A_{i,4},R_{j+m_1}].
\end{align*}
By \eqref{e1} and \eqref{e3}, we obtain
\begin{gather}
[A_{i,1},R_{j+m_1}]
=\frac{\hbar^2}{2}\sum\limits_{\substack{s,v\geq0}} \sum\limits_{u=-n_1}^{-1} E_{u,i}t^{-s}E_{i,j+m_1}t^{-v-1} E_{j+m_1,u}t^{s+v+1}\nonumber\\
\hphantom{[A_{i,1},R_{j+m_1}]=}{}\hspace{0.2mm}
-\frac{\hbar^2}{2}\sum\limits_{\substack{s,v\geq0}} \sum\limits_{u=-n_1}^{-1} E_{u,j+m_1}t^{-s-v-1}E_{j+m_1,i}t^{s+1}E_{i,u}t^{v},\label{AR1}\\
[A_{i,2},R_{j+m_1}]=0,\\
[A_{i,3},R_{j+m_1}]=0,\\
[A_{i,4},R_{j+m_1}]
=-\frac{\hbar^2}{2}\sum\limits_{\substack{s,v\geq0}} \sum\limits_{u=-n_1}^{-1} E_{i,j+m_1}t^{-s-1}E_{u,i}t^{s+1-v}E_{j+m_1,u}t^{v}\nonumber\\
\hphantom{[A_{i,4},R_{j+m_1}]=}{}\hspace{0.2mm} -\frac{\hbar^2}{2}\sum\limits_{\substack{s,v\geq0}} \sum\limits_{u=-n_1}^{-1} E_{i,j+m_1}t^{-s-v-1}E_{j+m_1,u}t^{s}E_{u,i}t^{v+1}\nonumber\\
\hphantom{[A_{i,4},R_{j+m_1}]=}{}\hspace{0.2mm} +\frac{\hbar^2}{2}\sum\limits_{\substack{s,v\geq0}} \sum\limits_{u=-n_1}^{-1}E_{i,u}t^{-v-1}E_{u,j+m_1}t^{-s}E_{j+m_1,i}t^{v+s+1}\nonumber\\
\hphantom{[A_{i,4},R_{j+m_1}]=}{}\hspace{0.2mm} +\frac{\hbar^2}{2}\sum\limits_{\substack{s,v\geq0}} \sum\limits_{u=-n_1}^{-1}E_{u,j+m_1}t^{-s}E_{i,u}t^{s-v-1}E_{j+m_1,i}t^{v+1}.\label{AR2}
\end{gather}
By the definition of $A_i$, we can rewrite $[A_i,S_{j+m_1}]$ as
\begin{align*}
[A_{i,1},S_{j+m_1}]-[A_{i,2},S_{j+m_1}]+[A_{i,3},S_{j+m_1}]-[A_{i,4},S_{j+m_1}].
\end{align*}
By a direct computation, we obtain
\begin{gather}
[A_{i,1},S_{j+m_1}]
=-\frac{\hbar^2}{2}\sum\limits_{\substack{s,v\geq0}} \sum\limits_{u=1}^{i} E_{u,i}t^{-v-s-1}E_{i,j+m_1}t^{v}E_{j+m_1,u}t^{s+1}\nonumber\\
\hphantom{[A_{i,1},S_{j+m_1}]=}{}\hspace{0.2mm}+\frac{\hbar^2}{2}\sum\limits_{\substack{s,v\geq0}} \sum\limits_{z=m_1+1}^{m_1+m_2} E_{z,j+m_1}t^{-v-s-1}E_{i,z}t^{v}E_{j+m_1,i}t^{s+1}\nonumber\\
\hphantom{[A_{i,1},S_{j+m_1}]=}{}\hspace{0.2mm}+\frac{\hbar^2}{2}\sum\limits_{\substack{s,v\geq0}} \sum\limits_{z=-n_1}^{-1} E_{z,j+m_1}t^{-v-s-1}E_{i,z}t^{v}E_{j+m_1,i}t^{s+1}\nonumber\\
\hphantom{[A_{i,1},S_{j+m_1}]=}{}\hspace{0.2mm}+\frac{\hbar^2}{2}\sum\limits_{\substack{s,v\geq0}} \sum\limits_{z=-n_1-n_2}^{-n_1-1} E_{z,j+m_1}t^{-v-s-1}E_{i,z}t^{v}E_{j+m_1,i}t^{s+1}\nonumber\\
\hphantom{[A_{i,1},S_{j+m_1}]=}{}\hspace{0.2mm}+\frac{\hbar^2}{2}\sum\limits_{\substack{s,v\geq0}} \sum\limits_{u=1}^{i} E_{u,j+m_1}t^{-s-1}E_{j+m_1,i}t^{-v} E_{i,u}t^{v+s+1}\nonumber\\
\hphantom{[A_{i,1},S_{j+m_1}]=}{}\hspace{0.2mm}-\frac{\hbar^2}{2}\sum\limits_{\substack{s,v\geq0}} \sum\limits_{z=m_1+1}^{m_1+m_2} E_{i,j+m_1}t^{-s-1}E_{z,i}t^{-v} E_{j+m_1,z}t^{v+s+1}\nonumber\\
\hphantom{[A_{i,1},S_{j+m_1}]=}{}\hspace{0.2mm}-\frac{\hbar^2}{2}\sum\limits_{\substack{s,v\geq0}} \sum\limits_{z=-n_1}^{-1} E_{i,j+m_1}t^{-s-1}E_{z,i}t^{-v} E_{j+m_1,z}t^{v+s+1}\nonumber\\
\hphantom{[A_{i,1},S_{j+m_1}]=}{}\hspace{0.2mm}-\frac{\hbar^2}{2}\sum\limits_{\substack{s,v\geq0}} \sum\limits_{z=-n_1-n_2}^{-n_1-1} E_{i,j+m_1}t^{-s-1}E_{z,i}t^{-v} E_{j+m_1,z}t^{v+s+1},\label{AS1}\\
[A_{i,2},S_{j+m_1}]
=\frac{\hbar^2}{2}\sum\limits_{\substack{s,v\geq0\\i>z}}  E_{i,z}t^{-s-v-1} E_{j+m_1,i}t^{s+1}E_{z,j+m_1}t^{v}\nonumber\\
\hphantom{[A_{i,2},S_{j+m_1}]=}{}\hspace{0.2mm}-\frac{\hbar^2}{2}\sum\limits_{\substack{s,v\geq0\\i>z}}  E_{j+m_1,z}t^{-v}E_{i,j+m_1}t^{-s-1}E_{z,i}t^{s+v+1},\label{AS2}\\
[A_{i,3},S_{j+m_1}]
=\frac{\hbar^2}{2}\sum\limits_{\substack{s,v\geq0}} \sum\limits_{u=1}^{i-1}E_{z,i}t^{-s-v-1} E_{i,j+m_1}t^{v}E_{j+m_1,u}t^{s+1}\nonumber\\
\hphantom{[A_{i,3},S_{j+m_1}]=}{}\hspace{0.2mm}-\frac{\hbar^2}{2}\sum\limits_{\substack{s,v\geq0}} \sum\limits_{u=1}^{i-1} E_{u,j+m_1}t^{-s-1}E_{j+m_1,i}t^{-v}E_{i,u}t^{v+s+1},\label{AS3}\\
[A_{i,4},S_{j+m_1}]
=\frac{\hbar^2}{2}\sum\limits_{\substack{s,v\geq0\\i<z}} (-1)^{p(z)} E_{i,z}t^{-v-s-1} E_{z,j+m_1}t^{v}E_{j+m_1,i}t^{s+1}\nonumber\\
\hphantom{[A_{i,4},S_{j+m_1}]=}{}\hspace{0.2mm}+\frac{\hbar^2}{2}\sum\limits_{\substack{s,v\geq0}} \sum\limits_{z=m_1+1}^{m_1+m_2}(-1)^{p(z)} E_{i,z}t^{-v-1} E_{z,j+m_1}t^{-s-1}E_{j+m_1,i}t^{s+v+2}\nonumber\\
\hphantom{[A_{i,4},S_{j+m_1}]=}{}\hspace{0.2mm}+\frac{\hbar^2}{2}\sum\limits_{\substack{s,v\geq0}} \sum\limits_{z=-n_1}^{-1}(-1)^{p(z)} E_{i,z}t^{-v-1} E_{z,j+m_1}t^{-s-1}E_{j+m_1,i}t^{s+v+2}\nonumber\\
\hphantom{[A_{i,4},S_{j+m_1}]=}{}\hspace{0.2mm}+\frac{\hbar^2}{2}\sum\limits_{\substack{s,v\geq0}} \sum\limits_{z=1}^{m_1}(-1)^{p(z)} E_{i,z}t^{-v-1} E_{z,j+m_1}t^{-s-1}E_{j+m_1,i}t^{s+v+2}\nonumber\\
\hphantom{[A_{i,4},S_{j+m_1}]=}{}\hspace{0.2mm}+\frac{\hbar^2}{2}\sum\limits_{\substack{s,v\geq0}} \sum\limits_{z=-n_1-n_2}^{-n_1-1}(-1)^{p(z)} E_{i,z}t^{-v-1} E_{z,j+m_1}t^{-s-1}E_{j+m_1,i}t^{s+v+2}\nonumber\\
\hphantom{[A_{i,4},S_{j+m_1}]=}{}\hspace{0.2mm}-\frac{\hbar^2}{2}\sum\limits_{\substack{s,v\geq0}} \sum\limits_{u=1}^{m_1}E_{i,j+m_1}t^{-v-1} E_{u,i}t^{-s}E_{j+m_1,u}t^{s+v+1}\nonumber\\
\hphantom{[A_{i,4},S_{j+m_1}]=}{}\hspace{0.2mm}+\frac{\hbar^2}{2}\sum\limits_{\substack{s,v\geq0}} \sum\limits_{u=1}^{m_1}E_{u,j+m_1}t^{-s-v-1}E_{i,u}t^{s}E_{j+m_1,i}t^{v+1}\nonumber\\
\hphantom{[A_{i,4},S_{j+m_1}]=}{}\hspace{0.2mm}-\frac{\hbar^2}{2}\sum\limits_{\substack{s,v\geq0}} \sum\limits_{z=-n_1}^{-1}(-1)^{p(z)} E_{i,j+m_1}t^{-s-v-2}E_{j+m_1,z}t^{s+1}E_{z,i}t^{v+1}\nonumber\\
\hphantom{[A_{i,4},S_{j+m_1}]=}{}\hspace{0.2mm}-\frac{\hbar^2}{2}\sum\limits_{\substack{s,v\geq0}} \sum\limits_{z=m_1+1}^{m_1+m_2}(-1)^{p(z)} E_{i,j+m_1}t^{-s-v-2}E_{j+m_1,z}t^{s+1}E_{z,i}t^{v+1}\nonumber\\
\hphantom{[A_{i,4},S_{j+m_1}]=}{}\hspace{0.2mm}-\frac{\hbar^2}{2}\sum\limits_{\substack{s,v\geq0}} \sum\limits_{z=1}^{m_1}(-1)^{p(z)} E_{i,j+m_1}t^{-s-v-2}E_{j+m_1,z}t^{s+1}E_{z,i}t^{v+1}\nonumber\\
\hphantom{[A_{i,4},S_{j+m_1}]=}{}\hspace{0.2mm}-\frac{\hbar^2}{2}\sum\limits_{\substack{s,v\geq0}} \sum\limits_{z=-n_1-n_2}^{-n_1-1}(-1)^{p(z)} E_{i,j+m_1}t^{-s-v-2}E_{j+m_1,z}t^{s+1}E_{z,i}t^{v+1}\nonumber\\
\hphantom{[A_{i,4},S_{j+m_1}]=}{}\hspace{0.2mm}-\frac{\hbar^2}{2}\sum\limits_{\substack{s,v\geq0\\i<z}} (-1)^{p(z)} E_{i,j+m_1}t^{-s-1}E_{j+m_1,z}t^{-v}E_{z,i}t^{s+v+1}.\label{AS4}
\end{gather}
Since
\begin{align*}
&\eqref{AS1}_1+\eqref{AS3}_1=-\frac{\hbar^2}{2}\sum\limits_{\substack{s,v\geq0}}  E_{i,i}t^{-s-v-1} E_{i,j+m_1}t^{v}E_{j+m_1,i}t^{s+1},\nonumber\\
&\eqref{AS1}_4+\eqref{AS3}_2=\frac{\hbar^2}{2}\sum\limits_{\substack{s,v\geq0}}  E_{i,j+m_1}t^{-s-1}E_{j+m_1,i}t^{-v}E_{i,i}t^{s+v+1},\nonumber\\
&-\eqref{AS2}_1-\eqref{AS4}_1=-\frac{\hbar^2}{2}\sum\limits_{\substack{s,v\geq0\\z\neq i}} (-1)^{p(z)} E_{i,z}t^{-v-s-1} E_{z,j+m_1}t^{v}E_{j+m_1,i}t^{s+1},\\
&-\eqref{AS2}_2-\eqref{AS4}_{12}=\frac{\hbar^2}{2}\sum\limits_{\substack{s,v\geq0\\z\neq i}} (-1)^{p(z)} E_{i,j+m_1}t^{-s-1}E_{j+m_1,z}t^{-v}E_{z,i}t^{s+v+1},
\end{align*}
we obtain
\begin{gather}
\eqref{AS1}_1+\eqref{AS3}_1+\eqref{AS1}_4+\eqref{AS3}_2-\eqref{AS2}_1-\eqref{AS4}_1-\eqref{AS2}_2-\eqref{AS4}_{12}\nonumber\\
\qquad=- \frac{\hbar^2}{2}\sum\limits_{\substack{s,v\geq0}} \sum\limits_{z=m_1+1}^{m_1+m_2}(-1)^{p(z)} E_{i,z}t^{-v-s-1} E_{z,j+m_1}t^{v}E_{j+m_1,i}t^{s+1}\nonumber\\
\hspace{11.75mm}-\frac{\hbar^2}{2}\sum\limits_{\substack{s,v\geq0}} \sum\limits_{z=-n_1-n_2}^{-n_1-1}(-1)^{p(z)} E_{i,z}t^{-v-s-1} E_{z,j+m_1}t^{v}E_{j+m_1,i}t^{s+1}\nonumber\\
\hspace{11.75mm}-\frac{\hbar^2}{2}\sum\limits_{\substack{s,v\geq0}} \sum\limits_{z=-n_1}^{-1}(-1)^{p(z)} E_{i,z}t^{-v-s-1} E_{z,j+m_1}t^{v}E_{j+m_1,i}t^{s+1}\nonumber\\
\hspace{11.75mm}-\frac{\hbar^2}{2}\sum\limits_{\substack{s,v\geq0}} \sum\limits_{z=1}^{m_1}(-1)^{p(z)} E_{i,z}t^{-v-s-1} E_{z,j+m_1}t^{v}E_{j+m_1,i}t^{s+1}\nonumber\\
\hspace{11.75mm}+\frac{\hbar^2}{2}\sum\limits_{\substack{s,v\geq0}} \sum\limits_{z=m_1+1}^{m_1+m_2}(-1)^{p(z)} E_{i,j+m_1}t^{-s-1}E_{j+m_1,z}t^{-v}E_{z,i}t^{s+v+1}\nonumber\\
\hspace{11.75mm}+\frac{\hbar^2}{2}\sum\limits_{\substack{s,v\geq0}} \sum\limits_{z=-n_1-n_2}^{-n_1-1}(-1)^{p(z)} E_{i,j+m_1}t^{-s-1}E_{j+m_1,z}t^{-v}E_{z,i}t^{s+v+1}\nonumber\\
\hspace{11.75mm}+\frac{\hbar^2}{2}\sum\limits_{\substack{s,v\geq0}} \sum\limits_{z=-n_1}^{-1}(-1)^{p(z)} E_{i,j+m_1}t^{-s-1}E_{j+m_1,z}t^{-v}E_{z,i}t^{s+v+1}\nonumber\\
\hspace{11.75mm}+\frac{\hbar^2}{2}\sum\limits_{\substack{s,v\geq0}} \sum\limits_{z=1}^{m_1}(-1)^{p(z)} E_{i,j+m_1}t^{-s-1}E_{j+m_1,z}t^{-v}E_{z,i}t^{s+v+1}.\label{AS}
\end{gather}
By a direct computation, we obtain
\begin{gather*}
\eqref{PR}_1-\eqref{PA1}_3+\eqref{AR1}_1-\eqref{AR2}_1-\eqref{AR2}_2+\eqref{AS1}_7-\eqref{AS4}_8\nonumber\\
\qquad=\frac{\hbar^2}{2}\sum\limits_{\substack{v\geq0}} \sum\limits_{u=-n_1}^{-1} E_{i,j+m_1}t^{-v-1}E_{j+m_1,u}E_{u,i}t^{v+1},\label{waa1}\\
\eqref{PR}_2-\eqref{PA1}_8+\eqref{AR1}_2-\eqref{AR2}_3-\eqref{AR2}_4+\eqref{AS1}_3-\eqref{AS4}_3\nonumber\\
\qquad=-\frac{\hbar^2}{2}\sum\limits_{\substack{v\geq0}} \sum\limits_{u=-n_1}^{-1}E_{i,u}t^{-v-1}E_{u,j+m_1}E_{j+m_1,i}t^{v+1},\\
\eqref{PS}_2-\eqref{PA3}_2-\eqref{AS4}_6-\eqref{AS4}_9\nonumber\\
\qquad=m_1\frac{\hbar^2}{2}\sum\limits_{\substack{s,v\geq0}} \sum\limits_{z=1}^{m_1}E_{i,j+m_1}t^{-s-v-2}E_{j+m_1,i}t^{s+v+2},\\
\eqref{PS}_3-\eqref{PA3}_5-\eqref{AS4}_4-\eqref{AS4}_7\nonumber\\
\qquad=-m_1\frac{\hbar^2}{2}\sum\limits_{\substack{s,v\geq0}} \sum\limits_{z=1}^{m_1}E_{i,j+m_1}t^{-s-v-2}E_{j+m_1,i}t^{s+v+2},
\\
\eqref{PS}_1-\eqref{PA3}_1-\eqref{AS4}_2+\eqref{AS}_1=0,\\
\eqref{PS}_1-\eqref{PA3}_6-\eqref{AS4}_{10}+\eqref{AS}_5=0,\\
-\eqref{QS}_1-\eqref{PA1}_7+\eqref{PA4}_2+\eqref{QA1}_2+\eqref{QA4}_1+\eqref{QA6}--\eqref{AS4}_5+\eqref{AS}_2=0,\\
-\eqref{QS}_2-\eqref{PA1}_1+\eqref{PA4}_4+\eqref{QA1}_1+\eqref{QA4}_2+\eqref{QA5}--\eqref{AS4}_{11}+\eqref{AS}_6=0,\\
\eqref{PA2}_1+\eqref{AS}_4=0,\qquad
\eqref{PA2}_2+\eqref{AS}_8=0,\\
-\eqref{PA1}_2+\eqref{AS1}_8=0,\qquad
-\eqref{PA1}_6+\eqref{AS1}_4=0,
\\
-\eqref{PA}_1+\eqref{AS1}_{6}
=\frac{\hbar^2}{2}\sum\limits_{\substack{s,v\geq0}}  \sum\limits_{\substack{u=m_1+1}}^{m_1+m_2}  E_{u,j+m_1}t^{-s-v-1}E_{j+m_1,u}t^{v+s+1},\\
-\eqref{PA}_2+\eqref{AS1}_{2}
=-\frac{\hbar^2}{2}\sum\limits_{\substack{s,v\geq0}} \sum\limits_{\substack{z=m_1+1}}^{m_1+m_2}  E_{z,j+m_1}t^{-s-v-1} E_{j+m_1,z}t^{s+v+1},\\
\eqref{PA4}_3+\eqref{AS}_3
=-\frac{\hbar^2}{2}\sum\limits_{\substack{s\geq0}} \sum\limits_{z=-n_1}^{-1}(s+1)(-1)^{p(z)} E_{i,z}t^{-s-1} E_{z,i}t^{s+1}\nonumber\\
\hphantom{\eqref{PA4}_3+\eqref{AS}_3=}{}\hspace{0.2mm}
+\frac{\hbar^2}{2}\sum\limits_{\substack{v\geq0}} \sum\limits_{u=-n_1}^{-1} E_{i,u}t^{-v-1}E_{u,j+m_1}E_{j+m_1,i}t^{v+1},\\
-\eqref{PA4}_5+\eqref{AS}_7
=\frac{\hbar^2}{2}\sum\limits_{\substack{s\geq0}} \sum\limits_{z=-n_1}^{-1}(s+1)(-1)^{p(z)} E_{i,z}t^{-s-1} E_{z,i}t^{s+1}\nonumber\\
\hphantom{-\eqref{PA4}_5+\eqref{AS}_7=}{}
-\frac{\hbar^2}{2}\sum\limits_{\substack{v\geq0}} \sum\limits_{u=-n_1}^{-1} E_{i,j+m_1}t^{-v-1} E_{j+m_1,u}E_{u,i}t^{v+1}.\label{waa2}
\end{gather*}
By adding \eqref{waa1}--\eqref{waa2}, we obtain \eqref{conclusion}. This completes the proof of Theorem~\ref{Commutativity}.

\section{Proof of Theorem~\ref{Tinf}}
This aapendix is devoted to the proof of Theorem~\ref{Tinf}.
We define a grading on $\mathfrak{b}$ by setting ${\rm deg}(e_{i,j})=\col(j)-\col(i)$. Let us set
\begin{align*}
f^{r}_{a,b}&=\sum\limits_{\substack{\col(i)=\col(j)+r\\\row(i)=a,\, \row(j)=b}}  e_{i,j}.
\end{align*}
For $a\in I_{u_1|q_1}$, we set $1\leq s_a\leq l$ as \smash{$a\in I_{u_1-u_{s_a}|q_1-q_{s_a}}\setminus I_{u_1-u_{s_a-1}|q_1-q_{s_a-1}}$}.
Since
\begin{align*}
&A=\bigl\{f^r_{a,b}\mid 0\leq r\leq l-1,s_a\leq s_b\bigr\}\cup\bigl\{f^r_{a,b}\mid s_a-s_b\leq r\leq l-1,s_a>s_b\bigr\}
\end{align*}
forms a basis of $\mathfrak{gl}(M|N)^f=\{g\in\mathfrak{gl}(M|N)|[f,g]=0\}$, it is enough to show that \smash{$W^{(1)}_{a,b}$} and~\smash{$W^{(2)}_{a,b}$} generate all terms of $A$ by \cite[Theorem~4.1]{KW1}. Here after, we set \smash{$f^r_{a,b}=0$} if $r\geq l$ or~${s_a>s_b}$, ${r<s_a-s_b}$

We show that \smash{$W^{(1)}_{a,b}$} and \smash{$W^{(2)}_{a,b}$} generate these terms by two claims, that is, Claims~\ref{T1} and~\ref{T3}. In Claim~\ref{T1} below, we show that \smash{$W^{(1)}_{a,b}$} and \smash{$W^{(2)}_{a,b}$} generate the term whose form is
\begin{equation*}
f^r_{i,j}[-1]+\text{higher terms}\qquad \text{for }i\neq j.
\end{equation*}
In Claim~\ref{T3} below, we prove that \smash{$W^{(1)}_{a,b}$} and \smash{$W^{(2)}_{a,b}$} generate the term whose form is
\begin{equation*}
f^r_{i,i}[-1]+\text{higher terms}.
\end{equation*}
\begin{Claim}\label{T1}\qquad
\begin{enumerate}
\item[$(1)$] The elements \smash{$W^{(1)}_{a,b}$} and \smash{$W^{(2)}_{a,b}$} generate the term whose form is
\begin{equation*}
f^r_{i,j}[-1]+\text{higher terms}\qquad \text{if }i\neq j.
\end{equation*}
\item[$(2)$] The elements \smash{$W^{(1)}_{a,b}$} and \smash{$W^{(2)}_{a,b}$} generate the term whose form is
\begin{align*}
&{(-1)}^{p(i)}f^r_{i,i}[-1]-{(-1)}^{p(i+1)}f^r_{i+1,i+1}[-1]+\text{higher terms}.
\end{align*}
\end{enumerate}
\end{Claim}
\begin{proof}
By a direct computation, the following equation holds:
\begin{align}
\bigl(f^x_{j,i}[-1]\bigr)_{(0)}f^w_{u,v}[-1]
=\delta_{i,u}f^{w+x}_{j,v}[-1]-\delta_{j,v}{(-1)}^{p(e_{i,j})p(e_{u,v})}f^{w+x}_{u,i}[-1]\label{nom}
\end{align}
if \smash{$f^x_{j,i}\neq0$} and \smash{$f^w_{u,v}\neq0$}.
Items (1) and (2) follow from \eqref{nom}.
We only show item~(1). Item~(2) can be proven by the same way as~\cite[Claim A.1.4]{U4}. The element \smash{$W^{(2)}_{a,b}$} has the form such that $f^1_{a,b}[-1]+\text{degree $0$ terms}$.
In the case that $s_a\leq s_b$, by \eqref{nom},
we obtain
\begin{align*}
\big(\big(W^{(2)}_{a,a}\big)_{(0)}\big)^rW^{(1)}_{a,b}=\big(f^1_{a,a}[-1])_{(0)}\big)^rW^{(1)}_{a,b}+\text{higher terms}=f^r_{a,b}+\text{higher terms}.
\end{align*}
In the case that $s_a=s_b+1$, by \eqref{nom}, we obtain
\begin{align*}
\big(\big(W^{(2)}_{a,a}\big)_{(0)}\big)^rW^{(2)}_{a,b}=\big(f^1_{a,a}[-1])_{(0)}\big)^rf^2_{a,b}+\text{higher terms}=f^{r+1}_{a,b}+\text{higher terms}.
\end{align*}
Thus, it is enough to show the case that $s_a>s_b+1$.

In the case $s_a>s_b+1$, we set $w_0$ as $a$ and take $w_u$ satisfying that $s_{w_u}=s_a-u$ for $1\leq u\leq s_a-s_b-1$. Then, by \eqref{nom}, we have
\begin{align*}
\prod_{u=1}^{s_a-s_b-1}\big(\big(f^1_{w_{u-1},w_u}[-1]\big)_{(0)}\big)f^{r-s_a+s_b+1}_{w_{s_a-s_b-1},b}=f^r_{a,b}.
\end{align*}
Thus, we have proved (1).
\end{proof}

\begin{Claim}\label{T3}
The elements \smash{$W^{(1)}_{i,j}$} and \smash{$W^{(2)}_{i,j}$} generate the term whose form is $f^r_{i,i}[-1]+\text{higher terms}$ for all $1\leq r\leq l-1$.
\end{Claim}
\begin{proof}
Suppose that $f^r_{j,j}[-1]+\text{higher terms}$ has been generated if $a_j\leq x-1$. If $a_i=x$, let us take $y$ satisfying $a_y=x-1$. Then, by \eqref{nom}, we have
\begin{align*}
\big(f_{i,k}^r[-1]\big)_{(0)}f_{k,i}^0[-1]=f_{i,i}^r[-1]-(-1)^{p(e_{i,k})}f^r_{k,k}[-1]
\end{align*}
for $1\leq r\leq x-1$. By the induction hypothesis, \smash{$(-1)^{p(e_{i,k})}f^r_{k,k}[-1]+\text{higher terms}$} is generated.~Thus, \smash{$W^{(1)}_{a,b}$} and \smash{$W^{(2)}_{a,b}$} generate the term whose form is $f^r_{i,i}[-1]+\text{higher terms}$.
\end{proof}

Since we complete the proof of Claims~\ref{T1} and \ref{T3}, we have proved Theorem~\ref{Tinf}.

\subsection*{Acknowledgements}
The author is grateful to the referees for careful reading of the manuscript and for many helpful comments that improved the paper. This work was supported by Japan Society for the Promotion of Science Overseas Research Fellowships, Grant Number JP2360303.


\pdfbookmark[1]{References}{ref}
\LastPageEnding

\end{document}